\documentclass[12pt]{amsart}

\usepackage{amssymb,amsbsy}
\usepackage{latexsym,array}
\usepackage{verbatim,color}
\usepackage{fullpage,euscript}

\usepackage[all]{xy}
\usepackage{xcolor}
\usepackage[colorlinks=true,linkcolor=blue,urlcolor=my_color,citecolor=redi]{hyperref}

\usepackage{tikz}
\usetikzlibrary{arrows,shapes,trees}

\definecolor{my_color}{rgb}{0,0.5,0.5}
\definecolor{MIXT}{rgb}{0.8,0.5,0.2}
\definecolor{darkblue}{rgb}{0,0.1,0.8}
\definecolor{redi}{rgb}{0.7,0,0}

\numberwithin{equation}{section}

\input {cyracc.def}
\font\tencyr=wncyr10 
\font\tencyi=wncyi10 
\font\tencysc=wncysc10 
\def\rus{\tencyr\cyracc}

\def\rusi{\tencyi\cyracc}
\def\rusc{\tencysc\cyracc}

\newcounter{rmke}
\numberwithin{rmke}{section}

\newtheorem{thm}{Theorem}[section]
\newtheorem{lm}[thm]{Lemma}
\newtheorem{cl}[thm]{Corollary}
\newtheorem{prop}[thm]{Proposition}

\theoremstyle{remark}
\newtheorem{rmk}[thm]{Remark}

\theoremstyle{definition}
\newtheorem{ex}[thm]{Example} 

\theoremstyle{plain}

\newcommand {\g}{{\mathfrak g}}

\newcommand {\h}{{\mathfrak h}}

\newcommand {\q}{{\mathfrak q}}

\newcommand {\es}{{\mathfrak s}}
\newcommand {\te}{{\mathfrak t}}

\newcommand{\gt}{\mathfrak}
\newcommand{\SL}{{\rm SL}}
\newcommand{\GL}{{\rm GL}}
\newcommand{\SO}{{\rm SO}}

\newcommand{\Spin}{{\rm Spin}}
\newcommand{\Sp}{{\rm Sp}}

\newcommand {\eus}{\EuScript}

\newcommand {\gS}{{\eus S }}

\newcommand {\vp}{\varphi}

\newcommand {\cs}{{\mathcal S}}

\newcommand{\Z}{{\mathbb Z}}

\newcommand {\md}{/\!\!/}

\newcommand {\ads}{{\mathrm{ad}^*}}
\newcommand {\Ad}{{\mathrm{Ad\,}}}

\newcommand {\Ann}{{\mathrm{Ann\,}}}

\newcommand {\ind}{{\mathrm{ind\,}}}
\newcommand {\Lie}{{\mathsf{Lie\,}}}

\newcommand {\rk}{{\mathrm{rk\,}}}

\newcommand {\spe}{{\mathsf{Spec\,}}}

\newcommand {\tr}{{\mathrm{tr\,}}}
\newcommand {\trdeg}{{\rm{tr.deg\,}}}

\newcommand {\tri}{\mathfrak{sl}_2}
\newcommand {\GR}[2]{{\textrm{{\sf\bfseries #1}}}_{#2}}

\newcommand {\un}{\underline}

\newcommand {\beq}{\begin{equation}}
\newcommand {\eeq}{\end{equation}}

\renewcommand{\le}{\leqslant}
\renewcommand{\ge}{\geqslant}

\newcommand{\brr}{\tilde{\boldsymbol{r}}}

\newcommand {\bbk}{\Bbbk}

\begin{document}
\setlength{\parskip}{3pt plus 2pt minus 0pt}
\hfill { {\color{blue}\scriptsize February 20, 2018}}
\vskip1ex

\title%
[Semi-direct products with free algebras of symmetric invariants]{Semi-direct products involving $\Sp_{2n}$ or $\Spin_n$\\ with free algebras of symmetric invariants}
\author[D.\,Panyushev]{Dmitri I. Panyushev}
\address[D.\,Panyushev]{Institute for Information Transmission Problems of the R.A.S, 
Bolshoi Karetnyi per. 19, Moscow 127051, Russia}
\email{panyushev@iitp.ru}
\author[O.\,Yakimova]{Oksana S.~Yakimova}
\address[O.\,Yakimova]{Universit\"at zu K\"oln,
Mathematisches Institut, Weyertal 86-90, 50931 K\"oln, Deutschland}
\email{yakimova.oksana@uni-koeln.de}
\thanks{The first author is partially supported by the RFBR grant {\rus N0} 16-01-00818. 
The second author is supported by a Heisenberg fellowship of the DFG}
\dedicatory{Dedicated to A.\,Joseph on the occasion of his 75th birthday}
\begin{abstract}
This is a part of an ongoing project, the goal of which is to classify all semi-direct products 
$\es=\g{\ltimes}V$ such that  $\g$ is a simple Lie algebra, $V$ is a $\g$-module, and  $\es$ has a free
algebra of symmetric invariants. 
In this paper, we obtain such a classification for the representations of the orthogonal and symplectic
algebras.
\end{abstract}
\maketitle

\section*{Introduction}

\noindent
Let $\bbk$ be a field with $\text{char\,}\bbk=0$. Let $S$ be an algebraic group defined over $\bbk$ with 
$\es=\Lie S$. The invariants of $S$ in the symmetric algebra $\gS(\es)=\bbk[\es^*]$ of $\es$ 
(=\,the symmetric invariants of $\es$ or of $S$) are denoted by $\bbk[\es^*]^S$ or $\gS(\es)^S$. If $S$ is 
connected, then we also  write $\bbk[\es^*]^\es$ or $\gS(\es)^\es$ for them.    

Let $\g$ be a reductive Lie algebra. Symmetric invariants of $\g$ over $\bar\bbk$ belong to the classical 
area of Representation Theory and Invariant Theory, where the most striking and influential results were 
obtained by Chevalley and Kostant in the 50s and 60s. Then pioneering insights of Kostant and Joseph revealed that the symmetric invariants of certain
non-reductive subalgebras of $\g$ can explicitly be described and that they are very helpful for understanding representations of $\g$ itself, see~\cite{jos77, jos11, kost}. 
This have opened a brave new world, full of adventures 
and hidden treasures. 
Hopefully, we have found (and presented here) some of them. 

Although the study of $\gS(\es)^S$ is hopeless in general, there are several classes of non-reductive 
algebras that are still tractable. One of them is obtained via a semi-direct product construction from finite-
dimensional representations of reductive groups, which is the main topic of this article,
see Section~\ref{sect:kos-th} below. Another interesting class of non-reductive algebras consists of 
truncated biparabolic  subalgebras~\cite{J}, see also ~\cite{FKs} and references therein. Yet another
class consists of the centralisers of nilpotent elements of $\g$, see~\cite{ppy}.
Remarkably, some truncated bi-parabolic subalgebras or centralisers occur also as semi-direct products. \\
In~\cite{Y}, the following problem has been 
proposed:  
\\ \indent
{\sl To classify the representations $V$ of {\bfseries simple}\/ algebraic groups $G$ with $\Lie G=\g$ such that 
the ring of symmetric invariants of the semi-direct product\/ $\es=\g\ltimes V$ is polynomial}. 
\\ \indent
It is easily seen that if $\es$ has this property, then $\bbk[V^*]^G$ is also a polynomial ring. (But not vice versa!) 
Therefore, the suitable representations $(G,V)$ are contained in
the list of ``coregular representations'' of simple algebraic groups, 
see~\cite{gerry1,ag79}. If a generic stabiliser for $(G,V)$ is trivial, then 
$\bbk[\es^*]^S\simeq \bbk[V^*]^G$. Therefore, it suffices to handle only ``coregular representations'' 
with non-trivial generic stabilisers. The latter can be determined with the help of Elashvili's 
tables~\cite{alela1}.
As it should have been expected, type ${\sf A}$  is the 
most difficult case. The solution for just one particular item, $V=m(\mathbb C^n)^*{\oplus}k\mathbb C^n$ 
for $G=\SL_n$, occupies the whole paper~\cite{Y}. 
This certainly means that obtaining classification in  the $\SL_n$-case  
requires considerable effort.  Although the results of~\cite{Y} are formulated over $\mathbb C$, we 
notice that they are actually valid over an arbitrary field of characteristic zero. 
The case of exceptional groups $G$ bas been considered in~\cite{exc}. The next logical step is to look at 
the symplectic and orthogonal groups $G$, which is done in this paper. 
To a great extent, our classification results rely on the theory developed by the second author in \cite{Y16}.

Let us give a brief outline of the paper. In Sections~\ref{sect:prelim}, we gather some properties of the 
arbitrary coadjoint representations, whereas in Section~\ref{sect:kos-th}, we stick to the coadjoint 
representations of semi-direct products and describe our classification techniques. After a brief interlude 
in Section~\ref{sect:1-A} devoted to an example in type ${\sf A}$, we dwell upon the classification of the 
suitable representations $V$ of the orthogonal (Section~\ref{sect:so}) and symplectic (Section~\ref{sect:sp}) groups.
Our results are summarised in Theorem~\ref{thm:main} and Tables~\ref{table:ort1},\,\ref{table-sp}. 
We are taking a somewhat unusual approach  towards a classification and  trying to present the 
essential ideas for {\bf all} pairs $(G,V)$ under consideration. Many pairs can be handled using general 
theorems presented in Section~\ref{sect:kos-th}, but some others require lengthy elaborated ad hoc 
considerations, see e.g. Theorem~\ref{th-D6-25}.
It appears {\sl a posteriori\/} that, for  {\bf all} representations $V$ of $G=\Sp_{2n}$ with polynomial ring 
$\bbk[V^*]^{\Sp_{2n}}$, the algebra of symmetric invariants $\gS(\es)^S$ is also polynomial. In most of 
the $\gt{sp}_{2n}$-cases, we explicitly describe the basic invariants. 
There is an interesting connection with the invariants of certain centralisers. 
In particular, if $V=\bbk^{2n}$ is the standard (defining) representations of $\Sp_{2n}$, then there is 
a kind of matryoshka-like structure between the  invariants of the semi-direct product and  
the symmetric invariants of the centraliser of the minimal nilpotent orbit in $\mathfrak{sp}_{2n-2}$.

{\sl \un{ Notation}.}
Let an algebraic group $Q$ act on an irreducible affine variety $X$. Then   $\bbk[X]^Q$ 
stands for the algebra of  $Q$-invariant regular functions on $X$ and $\bbk(X)^Q$
is the field of $Q$-invariant rational functions. If $\bbk[X]^Q$ is finitely generated, then 
$X\md Q:=\spe \bbk[X]^Q$. 
Whenever $\bbk[X]^Q$ is a graded polynomial ring, 
the elements of any set of algebraically independent homogeneous generators 
will be referred to as {\it basic invariants\/}. If $V$ is a $Q$-module and $v\in V$, then 
$\q_v=\{\xi\in\q\mid \xi{\cdot}v=0\}$ is the {\it stabiliser\/} of $v$ in $\q$ and 
$Q_v=\{g\in Q\mid g{\cdot}v=v\}$ is the {\it isotropy group\/}  of $v$ in $Q$.

Let $X$ be an irreducible variety (e.g. a vector space). 
We say that a property holds for ``generic $x\in X$" if  that property holds for all 
points of an open subset of $X$.  
An open subset is said to be {\it big}, if its complement does not contain divisors. 

Write $\gt{heis}_n$, $n\ge 0$, for the Heisenberg Lie algebra of dimension $2n{+}1$.

\section{Preliminaries on the coadjoint representations}
\label{sect:prelim}

\noindent
Let $Q$ be a connected algebraic group and $\q=\Lie Q$ . The {\it index\/}  of $\q$ is  
\[
\ind\q=\min_{\gamma\in\q^*} \dim\q_\gamma, 
\] 
where $\q_\gamma$ is the stabiliser of $\gamma$ in $\q$. 
In view of Rosenlicht's theorem~\cite[\S\,2.3]{VP}, $\ind\q={\rm tr.deg}\,\bbk(\q^*)^Q$.
If $\ind\q=0$,  then $\bbk[\q^*]^Q=\bbk$.
For a reductive $\g$, one has $\ind\gt g=\rk\gt g$. In this case, $(\dim\gt g+\rk\gt g)/2$ is the dimension 
of a Borel subalgebra of $\g$. For an arbitrary $\q$, set $\boldsymbol{b}(\q):=(\ind\q+\dim\q)/2$. 

One defines the  {\it singular set\/} $\q^*_{\sf sing}$ of $\q^*$ by 
$$
\q^*_{\sf sing}=\{\gamma\in\q^* \mid \dim\q_\gamma>\ind \q\}.
$$ 
Set also $\q^*_{\sf reg}:=\q^*\setminus \q^*_{\sf sing}$.
Further, $\q$ is said to have 
the ``codim--2" property (=\,to satisfy the ``codim--2" condition), if 
$\dim\q ^*_{\sf sing}\le \dim\q-2$.   
We say that $\q$ satisfies {\it the Kostant regularity criterion\/} (=\,{\it\bfseries KRC}) if the following properties hold for $\gS (\q)^Q$ and
$\xi\in\g^*$:
\begin{itemize}
\item $\gS (\q)^Q=\bbk[f_1,\dots,f_l]$ is a graded polynomial ring (with basic invariants $f_1,\dots,f_l$);
  \item $\xi\in\q^*_{\sf reg}$ if and only if $(\textsl{d}f_1)_\xi,\dots,(\textsl{d}f_l)_\xi$ are linearly
independent.
\end{itemize}
Every reductive Lie algebra has the ``codim--2" property and satisfies {\it\bfseries KRC}. 

Observe that $(\textsl{d}f)_\xi\in\q_\xi$ for each $f\in\bbk[\q^*]^Q$. 

\begin{thm}[{cf. \cite[Theorem~1.2]{coadj}}]\label{thm-codim2}
If $\q$ has the ``codim--2" property,  $\trdeg \gS (\q)^Q=\ind\q=l$, and
there are algebraically independent $f_1,\dots,f_l\in \gS (\q)^Q$ such that
$\sum\limits_{i=1}^l \deg f_i=\boldsymbol{b}(\q)$, then $f_1,\dots,f_l$ freely generate $\gS (\q)^Q$  
and the {\it\bfseries KRC\/} holds for $\q$.
\end{thm}

Suppose that $Q$ acts on an affine variety $X$. Then $f\in\bbk[X]$ is a {\it semi-invariant} of $Q$ if 
$g{\cdot}f\in\bbk f$ for each $g\in Q$. A semi-invariant is said to be {\it proper} if it is not an invariant.  
If $Q$ has no non-trivial characters (all $1$-dimensional representations of $Q$ are trivial), then 
it has no proper semi-invariants.   In particular, if $Q$ is a semi-direct product of a semisimple and a unipotent 
group, then all its semi-invariants are invariants.    We record a well-known observation:
\begin{itemize}
\item if $Q$ has no proper semi-invariants in $\gS(\q)$, then $\bbk(\q^*)^Q={\rm Quot}\,\bbk[\q^*]^Q$
and hence ${\rm tr.deg}\,\bbk[\q^*]^Q=\ind\q$. 
\end{itemize}

\begin{thm}[{cf. \cite[Prop.~5.2]{JSh}}]\label{thm-dif}
Suppose that  $Q$ has no proper semi-invariants in $\gS(\q)$ and 
$\gS(\q)^{Q}$ is freely generated by $f_1,\ldots,f_l$. Then 
the differentials $\textsl{d}f_1,\dots,\textsl{d}f_l$ are linearly independent on a 
big open subset of $\q^*$.  
\end{thm}

For any Lie algebra $\q$ defined over $\bbk$, set $\q_{\overline{\bbk}}:=\q{\otimes}_{\bbk}\overline{\bbk}$.\
Then
$\gS(\q_{\overline{\bbk}})^{\q_{\overline{\bbk}}}=\gS(\q)^{\q}{\otimes}_{\bbk}\overline{\bbk}$. 
If we extend the field, then a set of the generating  invariants over $\bbk$ is again a set of the generating invariants over 
$\overline{\bbk}$. In the other direction, having a minimal set ${\mathcal M}$ of homogeneous generators  over $\overline{\bbk}$, any $\bbk$-basis of $\left<{\mathcal M}\right>_{\overline{\bbk}}\cap \gS(\q)$ is a minimal set
of generators over $\bbk$.  The properties like ``being a polynomial ring'' do not change under field 
extensions. 
 The results in this paper are valid over fields that are not algebraically closed, but in the proofs 
we may safely assume that $\bbk=\overline{\bbk}$.

\section{On the coadjoint representations of a semi-direct product}
\label{sect:kos-th}

\noindent
For semi-direct products, there are some specific approaches to the  symmetric invariants.  
Our convention is that $G$ is always a connected {\bf reductive} group and $\gt g=\Lie G$, whereas a group $Q$ is not necessarily reductive and $\q=\Lie Q$. In this section, either 
$\es=\g{\ltimes} V$ or  $\es=\q{\ltimes} V$, where $V$ is a finite-dimensional $G$- or $Q$-module.
Then $S$ is a connected algebraic group with $\Lie S=\es$. For instance, $S=Q\ltimes \exp(V)$.

The vector space decomposition $\gt s=\q{\oplus} V$ leads to 
$\gt s^*=\q^*{\oplus} V^*$. For  $\q=\gt g$, we identify $\gt g$ with $\gt g^*$. 
Each element $x\in V^*$ is considered as a point of $\gt s^*$ that is zero on $\q$.
We have $\exp(V){\cdot}x=\ads(V){\cdot}x+x$, where each element of $\ads(V){\cdot}x$ is zero on $V$. 
Note that $\ads(V){\cdot}x\subset \Ann(\q_x)\subset \q^*$ and $\dim(\ads(V){\cdot}x)$ is equal to
$\dim(\ads(\q){\cdot}x)=\dim\q-\dim\q_x$. Therefore $\ads(V){\cdot}x= \Ann(\q_x)$. 

There is a general formula \cite{rais} for the index of $\gt s=\q{\ltimes} V$:   
\begin{equation}       \label{ind-sum}
\ind\gt s=\dim V-(\dim\q-\dim\q_x)+\ind\q_x \ \text{ with $x\in V^*$ generic. } 
\end{equation}

The decomposition $\gt s=\q{\oplus} V$ defines the bi-grading on $\gS(\gt s)$ and it appears that 
$\gS(\es)^{S}$ is a bi-homogeneous subalgebra, cf.~\cite[Theorem~2.3({\sf i})]{coadj}. 

For any  $x\in V^*$, the affine space $\q^*{+}x$ is $\exp(V)$-stable and $Q_x$-stable. Further, there is the restriction homomorphism
$$
\psi_x : \, \bbk[\gt s^*]^S\to \bbk[\q^*{+}x]^{Q_x{\ltimes}\exp(V)}\simeq \gS(\q_x)^{Q_x}.
$$
The existence of the isomorphism $\bbk[\q^*{+}x]^{\exp(V)}\simeq \gS(\q_x)$ is proven in \cite{Y16}. 
If we choose $x$ as the origin in $\gt q^*{+}x$, then actually $\psi_x(H)\in \gS(\gt q_x)$ 
for each $H\in \bbk[\gt s^*]^{\exp(V)}$, see \cite[Prop.~2.7]{Y16}.

Suppose that $Q\lhd  \tilde Q$ and there is an action of $\tilde Q$ on $V$ that extends the $Q$-action.  
Set $\tilde{\gt s}=\tilde{\q}{\ltimes} V$, $\tilde S=\tilde Q{\ltimes}\exp(V)$.

\begin{lm}\label{lm-sub} We have $\gS(\gt s)^{\tilde S}\subset \gS(\gt s)^{S}$ and 
$H\in \gS(\tilde{\gt s})^{\tilde S}$ lies in $\gS(\gt s)$ if and only if 
the restriction of $H$ to $\tilde{\q}^*{+}x$ lies in $\gS(\q_x)$ for a generic $x\in V^*$.
\end{lm}
\begin{proof}
The inclusion  $\gS(\gt s)^{\tilde S}\subset \gS(\gt s)^{S}$ is obvious. 
Now let $\gt m$ be  a vector space complement of $\q$ in $\tilde{\q}$. Then
$\gS(\tilde{\gt s})=\gS(\q)\otimes\gS(\gt m)\otimes\gS(V)$.  If $H$ does not lie in 
$\gS(\q)\otimes\gS(V)$, then $H|_{\tilde{\q}^*+x}$ does not lie in $\gS(\q)$ for any $x$ from 
a non-empty open subset of $V^*$. 

Finally, suppose that $H\in \gS(\gt s)^{\exp(V)}$. Then $H|_{\tilde{\q}^*+x}$ lies in 
$\gS(\tilde{\q}_x)$ by \cite[Prop.~2.7]{Y16}. Clearly, $\gS(\q)\cap \gS(\tilde{\q}_x)=\gS(\q_x)$.   
\end{proof}

\begin{prop}[Prop.\,3.11 in \cite{Y16}]         \label{non-red} 
Let $Q$ be a connected algebraic group acting on a finite-dimensional vector space $V$. 
Set $\es=\q\ltimes V$. Suppose that 
$Q$ has no proper semi-invariants in\/ $\bbk[\es^*]^{\exp(V)}$ and\/ $\bbk[\es^*]^S$ is a 
polynomial  ring in $\ind \es$ variables. 
For  generic $x\in V^*$, we then have 
\begin{itemize}
\item the restriction map 
$\psi\!: \bbk[\es^*]^S \to \bbk[\q^*{+}x]^{Q_x{\ltimes}\exp(V)}\simeq \gS(\q_x)^{Q_x}$ is onto; 
\item $\gS(\q_x)^{Q_x}$ coincides with $\gS(\q_x)^{\q_x}$; 
\item $\gS(\q_x)^{Q_x}$ is a polynomial ring in $\ind\q_x$ variables. 
\end{itemize}
\end{prop}
\noindent
Note that $Q$ is not assumed to be reductive and $Q_x$ is not assumed to be connected in the above proposition!

Let now $V$ be a $G$-module. 
By a classical result of Richardson, there is a non-empty  open subset $\Omega\subset V^*$ such that 
the stabilisers $G_x$  are conjugate in $G$ for all $x\in \Omega$, see e.g. \cite[Theorem~7.2]{VP}. 
In this situation (any representative of the conjugacy class of) $G_x$
is called a {\it generic isotropy group}, denoted $\mathsf{g.i.g.}(G:V^*)$, and $\g_x=\Lie G_x$
is a  {\it generic stabiliser} for the $G$-action on $V^*$.

If $G$ is semisimple and $V$ is a reducible $G$-module, say $V=V_1\oplus V_2$, 
then there is a trick that allows us to relate the polynomiality property for the symmetric invariants of 
$\es=\g\ltimes V$ to a smaller semi-direct product. The precise statement is as follows.

\begin{prop}[{cf. \cite[Prop.~3.5]{exc}}]   \label{prop:trick}
With $\es=\g\ltimes (V_1\oplus V_2)$ as above, let $H$ be a generic isotropy group for $(G:V^*_1)$. 
If\/ $\bbk[\es^*]^S$ is a polynomial ring, then so is\/ $\bbk[\tilde\q^*]^{\tilde Q}$, where 
$\tilde Q=H{\ltimes}\exp(V_2)$ or  $H^\circ{\ltimes}\exp(V_2)$. 
\end{prop}

The above passage from $\es$ to $\tilde\q$, i.e., from $(G, V_1\oplus V_2)$ to $(H^\circ, V_2)$ is 
called a {\it reduction}, and we denote it by $(G, V_1\oplus V_2)\longrightarrow (H^\circ, V_2)$ in the diagrams below.
This proposition is going to be used as a tool for proving that $\bbk[\es^*]^S$ is not polynomial.

In what follows, the irreducible representations of simple groups are often identified with their highest 
weights, using the Vinberg--Onishchik numbering of the fundamental weights~\cite{VO}. 
For instance, if $\varphi_1,\dots,\varphi_n$ are the fundamental weights of a simple algebraic group $G$, then $V=\varphi_i+2\varphi_{j}$ stands for the direct sum of three simple $G$-modules, with highest weights $\varphi_i$ (once) and $\varphi_{j}$ (twice). A full notation is $V=V_{\vp_i}+2V_{\vp_i}$.
Note that adding a trivial $1$-dimensional $G$-module $\bbk$ to $V$ does not affect the polynomiality property for $\es$.

\begin{ex}\label{ex-ch-1}
There is a diagram (tree) of reductions:
\[
\xymatrix{
(\Spin_{11}, 2\varphi_1{+}\varphi_5)\ar[dr]  & (\Spin_{10},\varphi_1{+}\varphi_4{+}\varphi_5) \ar[d] & 
(\Spin_{10},\varphi_1{+}2\varphi_4) \ar[ld]  & (\Spin_8, 2\varphi_1{+}\varphi_3) \ar[dl] \\
(\Spin_{12}, 3\varphi_1{+}\varphi_5) \ar[r] & 
(\Spin_{9}, 2\varphi_4)\ar[r] & (\Spin_{7}, \varphi_1{+}\varphi_3{+}\bbk) \ar[r]
 &  \text{\framebox{$(\SL_4, \varphi_1{+}\varphi_1^*$)} }. }
\]
For instance, the first diagonal arrow means that for $G=\Spin_{11}$ and $V_1=2\vp_1$, we have
$\mathsf{g.i.g.}(G,V_1)=\Spin_9$ and the restriction of $V_2=\vp_5$ to $H=\Spin_9$ is the $H$-module
$2\vp_4$. 
The terminal item (in the box) does not have  the polynomiality property by \cite{Y}.
Therefore  all the items 
here do not have the polynomiality property by Proposition~\ref{prop:trick}. 
\end{ex}

The action $(G:V)$ is said to be {\it stable\/} if the union of closed $G$-orbits is dense in $V$. Then 
 $\mathsf{g.i.g.}(G:V)$ is necessarily reductive.

We mention the following good situation. Suppose that  $G$ is semisimple.  If a generic stabiliser for the $G$-action on $V^*$ is reductive, then the action $(G:V^*)$ is stable~\cite[\S\,7]{VP}.
 Moreover,  $S$ has only trivial characters and no proper semi-invariants.

\begin{ex}[{cf. \cite[Example~3.6]{Y16}}] \label{ex-A1} If $G$ is semisimple, $\gt g_x=\gt{sl}_2$ for 
$x\in V^*$ generic, and $\bbk[V^*]^G$ is a polynomial ring, then 
$\gS(\gt s)^{\gt s}$ is a polynomial ring. 
\end{ex}  

We say that $\dim V\md G$ is the {\it rank} of the pair $(G,V)$.
For $(G,V^*)$ of rank one, we have two general results.

Consider the following assumptions on $G$ and $V$:

\begin{itemize}
\item[($\diamondsuit$)] \  
the action  $(G:V^*)$ is stable, $\bbk[V^*]^G$ is a  polynomial ring, $\bbk[\g^*_\xi]^{G_\xi}$ is a polynomial 
ring for generic $\xi\in V^*$, and $G$ has no proper semi-invariants in $\bbk[V^*]$.
\end{itemize}

\begin{thm}[ {\cite[Theorem~2.3]{exc} }]     \label{V-rank-1}
Suppose that $G$ and $V$ satisfy condition $(\diamondsuit)$ and $V^*\md G=\mathbb A^1$, i.e.,
$\bbk[V^*]^G=\bbk[F]$ for some homogeneous $F$. Let $L$ be a generic isotropy group for $(G:V^*)$. 
Assume further that $D=\{x\in V^*\mid F(x)=0\}$ contains an open $G$-orbit, say $G{\cdot} y$, 
$\ind\g_y=\ind\gt l=:\ell$, and $\gS(\g_y)^{G_y}$ is a polynomial ring in $\ell$ variables  
with the same degrees of generators as $\gS(\gt l)^{L}$. 
Then $\bbk[\es^*]^S$ is a polynomial ring in $\ind\es=\ell+1$ variables.
\end{thm}

\begin{lm}\label{l-rk1}
Suppose that $G$ is semisimple, $\bbk[V^*]^G=\bbk[F]$ and a generic isotropy group for $(G:V^*)$, 
say $L$, is connected and is either of type $\GR{B}{2}$ or $\GR{G}{2}$. Then $\es=\g\ltimes V$ has the 
polynomiality property.   
\end{lm}
\begin{proof}
Let $x\in V^*$ be generic and $G_x=L$, hence $\g_x=\gt l$. By~\cite[Lemma~3.5]{Y16}, there are irreducible 
bi-homogeneous $S$-invariants $H_1$ and $H_2$ such that their restrictions to $\g+x=\gt g^*+x$ yield 
the basic symmetric invariants of $\gt l$ under the isomorphism 
$ \bbk[\g^*{+}x]^{G_x{\ltimes}\exp(V)}\simeq \gS(\g_x)^{G_x}$.
Furthermore, $\bbk[\es^*]^{S}=\bbk[F,H_1,H_2]$ if and only if 
$H_1$ and $H_2$ are algebraically independent over $\bbk[D]^G=\bbk$ 
on $\g\times D$, where $D$ is the zero set of $F$. 
W.l.o.g., we may assume that $\deg_{\g}H_1=2$ and $\deg_{\g}H_2=4$ (if $L=\GR{B}{2}$) or
$\deg_{\g}H_2=6$ (if $L=\GR{G}{2}$).  
We may also assume that a non-trivial relation among 
$H_1|_{\g{\times}D}$, $H_2|_{\g{\times}D}$ is homogeneous w.r.t. $\g$ and therefore boils down to 
$\frac{H_1^\alpha}{H_2}\equiv a \!\mod\!(F)\,$ for $\alpha\in\{2,3\}$, depending on $L$, and $a\in\bbk$.  Such a relation means that $H_2$ is chosen wrongly and has to be replaced by a polynomial 
$(H_2-aH_1^\alpha)/F^r$ with the largest possible $r\ge 1$.  This modification decreases the total degree of $H_2$ and hence it cannot be performed infinitely many times. 
\end{proof}

The following result holds for actions of arbitrary rank. 

\begin{lm}\label{l-A2}
Suppose that $G$ is semisimple, $\bbk[V^*]^G$ is a polynomial ring and a generic isotropy group for  
$(G:V^*)$  is a connected group of type {\sf A}$_2$.  Assume further that, for any $G$-stable divisor 
$D\subset V^*$ and a generic point $y\in D$, we have 
$\dim\gS^2(\gt g_y)^{G_y}=\dim\gS^3(\gt g_y)^{G_y}=1$ and that these unique (up to a scalar) invariants 
are algebraically independent.  Then $\gt s=\gt g\ltimes V$ has the polynomiality property.  
\end{lm}
\begin{proof}
The statement readily follows from  \cite[Lemma~3.5]{Y16}. 
\end{proof}

\subsection{Yet another case of a surjective restriction}  
By Proposition~\ref{non-red}, if $x\in V^*$ is generic, then the restriction homomorphism 
$\psi_x : \bbk[\gt s^*]^S\to\bbk[\q^*{+}x]^{Q_x{\ltimes}\exp(V)}$ is surjective, whenever $\bbk[\gt s^*]^S$ is a a polynomial ring and $Q$ has no proper semi-invariants in
$\bbk[\gt s^*]^{\exp(V)}$.
On the other hand, $\psi_x$ is surjective for generic $x\in V^*$ if $Q=G$ is reductive and the $G$-action on $V^*$ is stable~\cite[Theorem~2.8]{Y16}.
It is likely that the surjectivity holds for a wider class of semi-direct products. 

Suppose that $\bbk$ is algebraically closed. Take $Q$ and $V$ such that 
$\dim(Q{\cdot}\xi)=\dim Q-1$ for generic $\xi\in V^*$. Assume that $\bbk[V^*]^Q\ne \bbk$. Then 
$\bbk[V^*]^Q=\bbk[F]$, where $F$ is a homogeneous polynomial of degree $N\ge 1$,   
 $\bbk(V^*)^Q=\bbk(F)$, and $F$ separates generic $Q$-orbits on $V^*$. Hence 
$\bbk \xi \cap Q{\cdot}\xi=\{ax\mid a\in\bbk, a^N=1\}$ for generic $\xi\in V^*$. 
Let $N_Q(\bbk \xi)$ be the normaliser of the line $\bbk \xi$. Then 
$N_Q(\bbk \xi)=C_N{\times}Q_\xi$, where $C_N\subset\bbk^{\!^\times}\!$
is a cyclic group of order $N$.  Let $C_N$ act on $V$ faithfully, then 
$\tilde Q:=C_N{\times} Q$ acts on $V$ and $\tilde Q_\xi \simeq C_N{\times}Q_\xi$. 
If $H\in\bbk[\gt s^*]^Q$ is homogeneous in $V$, then 
$\psi_\xi(H)$ is an eigenvector of $C_N\subset \tilde Q_\xi$ and the corresponding 
eigenvalue depends only on $\deg_V H$. 

\begin{thm}[Generalised surjectivity or the ``rank-one argument"] \label{sur-1}
Let $Q$ be a connected algebraic group with $\Lie Q=\q$.
Suppose that $V$ is  a $Q$-module such that $Q$ has no proper semi-invariants in $\bbk[V^*]$ 
and\/ $\bbk[V^*]^Q=\bbk[F]$ with $F\not\in\bbk$. Set $\es=\q{\ltimes} V$, $S=Q{\ltimes}\exp(V)$. Then 
the natural homomorphism
\[
   \psi_\xi: \bbk[\es^*]^S \to \bbk[\q^*+\xi]^{Q_\xi{\ltimes}\exp(V)}\simeq\gS(\q_\xi)^{Q_\xi}
\]
is onto for generic $\xi\in V^*$. Moreover, 
if $h\in\bbk[\q^*{+}\xi]^{Q_\xi{\ltimes}\exp(V)}$ is a semi-invariant of $N_{Q}(\bbk\xi)$, then there is 
a homogeneous in $V$ polynomial $H\in\bbk[\gt s^*]^S$ with $\psi_\xi(H)=h$.  
\end{thm}
\begin{proof}
Let $S$ act on an irreducible variety $X$. A classical result of Rosenlicht~\cite[\S\,2.3]{VP} implies 
that the functions $f_1,\ldots,f_m\in\bbk(X)^S$ generate 
$\bbk(X)^S$ if and only if they separate generic $S$-orbits on $X$. Let 
$U\subset \gt s^*$ be a non-empty open subset such that for every two different orbits 
$S{\cdot}u, S{\cdot}u'\subset U$, there is ${\bf f}\in\bbk(\gt s^*)^S$ separating them, meaning that  
${\bf f}$ takes  finite values at $u,u'$ and  ${\bf f}(u)\ne {\bf f}(u')$. 
Then $U\cap (\q^*{+}\xi)\ne \varnothing$ for generic $\xi\in V^*$ and hence generic 
$Q_\xi{\ltimes}\exp(V)$-orbits on $\q^*{+}\xi$ are separated by rational $S$-invariants for any such $\xi$. 
In other words, for every $h\in \bbk(\q^*{+}\xi)^{Q_x{\ltimes}\exp(V)}$ there is ${\brr}\in\bbk(\gt s^*)^S$ 
such that $\psi_{\xi}({\brr}):={\brr}\vert_{\q^*+\xi}=h$. 

The same principle applies to the group $\bbk^{\!^\times}\!{\times}S$, where $\bbk^{\!^\times}\!$ acts 
on $V$ by $t{\cdot}v=tv$ for all $t\in \bbk^{\!^\times}\!$, $v\in V$. 
A rational invariant of $(\bbk^{\!^\times}\!{\times}Q)_{\xi}{\ltimes}\exp(V)$ on $\q^*{+}\xi$ extends to a 
rational $(\bbk^{\!^\times}\!{\times}S)$-invariant on $\gt s^*$. 

The absence of proper semi-invariants implies that $\bbk(V^*)^Q=\bbk(F)$. Hence a generic 
$Q$-orbit on $V^*$ is of dimension $\dim V{-}1$. Assume that $F$ is homogeneous and set $N:=\deg F$. 

Choose a generic point $\xi\in V^*$ with $F(\xi)\ne 0$ and with $\dim(Q{\cdot}\xi)=\dim V{-}1$.  
Then $N_Q(\bbk \xi)=C_N{\times}Q_\xi$. As above, set 
 $\tilde Q:=C_N{\times}Q$ and also $\tilde S:=C_N{\times}S$. We regard $\tilde Q$ as a subgroup of $\bbk^{\!^\times}\!{\times}Q$.  Now $\tilde Q_\xi=(\bbk^{\!^\times}\!{\times}Q)_\xi$. 

The group $C_N\subset \tilde Q_\xi$ acts on $\bbk[\q^*+\xi]^{Q_\xi{\ltimes}\exp(V)}$ and this action is diagonalisable.  
Suppose that $h\in \bbk[\q^*+\xi]^{Q_\xi{\ltimes}\exp(V)}$ is an eigenvector of $C_N$.  
First we show that there is  $r\in\bbk(\gt s^*)^S$ such that $\psi_\xi(r)$ is an 
eigenvector  of $C_N\subset \tilde Q_\xi$ with the same weight as $h$.  

Recall that $h$ extends to a rational $S$-invariant ${\brr}\in\bbk(\gt s^*)^S$. 
The group $C_N$ is finite, hence $\brr$ is contained in a finite-dimensional $C_N$-stable 
vector space and thereby $\brr$ is a sum of rational $S$-invariant $C_N$-eigenvectors. 
Since a copy of $C_N$ sitting in $\tilde Q$ stabilises $\xi$, we can replace $\brr$ 
with a suitable  $C_N$-semi-invariant  component. 
By a standard 
argument, this new $\brr$ is a ratio of  two regular $\tilde S$-semi-invariants, say
$\brr=q/f$ now.  Each bi-homogenous w.r.t. $\gt s=\q{\oplus}V$ component  
of $q$ (or $f$) is again a semi-invariant of $\tilde S$ of the same weight as $q$ (or $f$). 
Let us replace $f$ (and $q$) with any of its non-zero bi-homogenous components. 
The resulting rational function $r$ has the same weight as $\brr$.  
In particular, $r$ is an  $S$-invariant. 
Thus, we have found the required rational function. 
Since $r$ is a semi-invariant of $\bbk^{\!^\times}\!$, it is defined on 
a non-empty open subset of $\q^*{\times} Q{\cdot}x$ for each $x\in V^*$ such that  $F(x)\ne 0$
and $\dim(Q{\cdot}x)=\dim V{-}1$. 

Set ${\bar r}:=\psi_\xi(r)\in\bbk(\q^*{+}\xi)$.  
Then $h/\bar{r}\subset \bbk(\q^*{+}\xi)^{\tilde Q_\xi{\ltimes}\exp(V)}$ and therefore 
extends to a rational $(\bbk^{\!^\times}\!{\times}S)$-invariant on $\gt s^*$. Multiplying 
the extension by $r$, we obtain a rational $S$-invariant $R$, which is also an eigenvector of 
$\bbk^{\!^\times}\!$. Let $R=H/P$, where $H,P\in\bbk[\gt s^*]$ are relatively prime. Then both $H$ and $P$ are homogenous in $V$. 
Note that $R$ is defined on $\q^*{+}\xi$, therefore also on
$\q^*{\times} Q{\cdot}\xi$ and finally on $\q\times \bbk^{\!^\times}\!(Q{\cdot}\xi)$, because 
$R(\eta{+}a\xi)=a^kR(\eta{+}\xi)$ for some $k\in\Z$ and for all $a\in \bbk^{\!^\times}\!$, $\eta\in\q^*$.   
Hence $P$ is 
a polynomial in $F$, more explicitly, $P=F^d$ fore  some $d\ge 0$. 
Multiplying $R$ by $\frac{F^d}{F(\xi)^d}$ yields the required pre-image $H$. 
 \end{proof}

\begin{rmk}
Since $\bbk[V^*]^Q=\bbk[F]$ and there are no proper $Q$-semi-invariants in $\bbk[V^*]$, 
$\q^*{\times} Q{\cdot}\xi$ is a big open subset of
\[
Y_\alpha=\{\q^*{+}x \mid F(x)=F(\xi)\}=\{\gamma\in\gt s^* \mid F(\gamma)=\alpha\} ,
\]  
where $\alpha=F(\xi)$. For a reductive group $G$, one knows that any regular $G$-invariant on a closed 
$G$-stable subset $Y\subset X$ of an affine  $G$-variety $X$ extends to a regular $G$-invariant on $X$. 
Assuming that the image of $Q$ in $\GL(V^*)$ is reductive, we could  present a different proof 
of Theorem~\ref{sur-1}, similar to the proof of Theorem~2.8 in \cite{Y16}. 
\end{rmk}

\subsection{Tables and classification tools}
Our goal is to classify the pairs $(G,V)$ such that $G$ is either $\Spin_n$ or $\Sp_{2n}$ and the 
semi-direct product $\es=\g\ltimes V$ has a {\sf F}ree {\sf A}lgebra of symmetric invariants,
({\sf FA}) for short.  We also say that $(G,V)$ is a {\it positive} (resp. {\it negative}) case, if the property
({\sf FA}) {\bf is} (resp. {\bf is not}) satisfied for $\es$.
 
\begin{ex} \label{ex:adjoint}
If $G$ is arbitrary semisimple, then $\g\ltimes\g^{\sf ab}$, where $\gt g^{\sf ab}$ is an Abelian ideal
isomorphic to $\gt g$ as a $\gt g$-module,   always has  ({\sf FA})~\cite{takiff}.
Therefore we exclude the adjoint representations from our further consideration. 
\end{ex}

\textbullet \ \ If $\bbk[\es^*]^S$ is a polynomial ring, then so is $\bbk[V^*]^G$~\cite[Section\,2\,(A)]{coadj} 
(cf. \cite[Section~3]{Y16}). For this reason, we only have to examine all representations of $G$ with 
polynomial rings of invariants.

\textbullet \ \ Since the algebras $\bbk[V]^G$ and $\bbk[V^*]^G$ (as well as $\gS(\g\ltimes V)^{G\ltimes V}$ and 
$\gS(\g\ltimes V^*)^{G\ltimes V^*}$) are isomorphic, it suffices to keep track of either $V$ or $V^*$. 
The same principle applies to the two half-spin representations in type $\GR{D}{2m}$.   

\begin{ex}   \label{ex:h=0}
If a generic stabiliser for  $(G{:}V^*)$ is trivial, then $\bbk[\es^*]^S\simeq \bbk[V^*]^G$~\cite[Theorem\,6.4]{p05} (cf.~\cite[Example\,3.1]{Y16}). Therefore all such semi-direct products have   ({\sf FA}).
\end{ex}

We are lucky that there is a  classification of the 
representations of the simple algebraic groups with non-trivial generic stabilisers obtained by 
A.G.\,Elashvili~\cite{alela1}. In addition, the two independent
classifications in \cite{gerry1,ag79} provide the list of representations of simple algebraic groups
with polynomial rings of invariants. Combining them, we obtain the representations in 
Tables~\ref{table:ort1} and \ref{table-sp}. 

{\it \bfseries Explanations to the tables.} As in~\cite{ag79, alela1,Y16, md-ko,exc},
we use the Vinberg--Onishchik numbering of fundamental weights, see~\cite[Table\,1]{VO}. In both 
tables, $\h$ is a generic stabiliser for $(G:V)$ and the last column indicates whether ({\sf FA}) is satisfied 
for $\es$ or not.  Naturally, the positive cases are marked with $`+'$. 
This last column represents the main results of the article.
The ring $\bbk[V^*]^G$ is always a polynomial ring in $\dim V\md G$ variables. If the expression for 
$\dim V\md G$ is bulky, then it is not included in Table~\ref{table:ort1}. However, one always has 
$\dim V\md G=\dim V-\dim G+\dim\h$. 
If $\es$ has ({\sf FA}), then $\ind\es=\dim V\md G + \ind\h$ is the total number of the basic invariants in
$\bbk[\gt s^*]^S$. The symbol $\GR{U}{n}$ in Table~\ref{table-sp} stands for a commutative Lie subalgebra of dimension $n$ that consists of nilpotent elements.

Our classification is summarised in  the following
\begin{thm}   \label{thm:main}
Let $G$ be either $\Spin_n$ or $\Sp_{2n}$, $V$ a finite-dimensional rational $G$-module, and
$\es=\g\ltimes V$. Then $\bbk[\es^*]^S$ is a free algebra if and only if one of the following conditions is satisfied:
\begin{itemize}
\item[\sf (i)] \  $V=\g$;
\item[\sf (ii)] \  $V$ or $V^*$  occurs in Tables~\ref{table:ort1} and \ref{table-sp}, and the last column
is marked with `$+$'. It is also possible to permute $\vp_5$ and $\vp_6$ for\/ $\GR{D}{6}$, and take any permutation of $\vp_1,\vp_3,\vp_4$ for\/ $\GR{D}{4}$.
\item[\sf (iii)] \  $\bbk[V]^G$ is a free algebra and $\mathsf{g.i.g.}(G:V)$ is finite, i.e., $(G,V)$ is contained in the lists of \cite{gerry1,ag79}, but is not contained in the tables of\/ \cite{alela1}.
\end{itemize}
\end{thm} 

\begin{table}[htbp]
\caption{The representations of the orthogonal groups with polynomial ring $\bbk[V]^G$ and non-trivial generic stabilisers}   \label{table:ort1}   
\begin{center}
\begin{tabular}{>{\sf}c<{}>{$}c<{$}>{$}c<{$}>{$}c<{$}>{$}c<{$}>{$}c<{$}>{$}c<{$}cc}
{\rus N0} &  G & V &  \dim V & \dim V\md G &  \h & \ind\es & ({\sf FA}) \\ \hline \hline
1 &  \SO_n &m\vp_1,\, m {<} n{-}1  & mn & {m(m+1)\over 2}  & \mathfrak{so}_{n-m} & {m+1 \choose 2}{+}[\frac{n-m}{2}] & $+$  \\[.5ex]  \hline
2a & \GR{B}{3} &  \vp_3 & 8 & 1 & \GR{G}{2} & 3 & $+$ \\
2b &  &   \tabcolsep=0.2em \begin{tabular}{c} $m\vp_1{+}m'\vp_3$ \\ $2{\le} m{+}m'{\le} 3$ \\ $m'{>}0$
\end{tabular}
 & 7m{+}8m'& & \!\GR{A}{4-m-m'}\! & & $-$, if $(1,1)$ 
\\[.5ex]  \hline  
3a & \GR{B}{4} & \vp_4 & 16 & 1 &\GR{B}{3} & 4 & $+$ \\
3b &   & \vp_1{+}\vp_4 & 25 & 3 &\GR{G}{2} & 5 & $+$ \\
3c &   &  2\vp_1{+}\vp_4 & 34 & 6 &\GR{A}{2} & 8 & $+$ \\
3d &   &  3\vp_1{+}\vp_4 & 43 & 10 & \GR{A}{1} & 11 & $+$ \\
3e &   &  2\vp_4 & 32 & 4 & \GR{A}{2} & 6 & $-$ \\
3f &   &  \vp_1{+}2\vp_4 & 41 & 8 & \GR{A}{1} & 9 & $+$
\\[.5ex]  \hline
4 & \GR{B}{5} & \tabcolsep=0.2em \begin{tabular}{c} $m\vp_1{+}\vp_5$, \rule{0pt}{4mm}\\ $0{\le}m{\le} 3$
\end{tabular}
& 32{+}11m & 1{+}m{+}m^2 &\GR{A}{4-m} & 5{+}m^2 & 
\tabcolsep=0.2em \begin{tabular}{c} $+$, if $m=0{,}3$ \\ $-$, if $m=1{,}2$ \end{tabular}
\\[.5ex]  \hline
5a & \GR{B}{6} & \vp_6 & 64 & 2 &\GR{A}{2}{+}\GR{A}{2} & 6 & $+$ \\  
5b &  & \vp_1{+}\vp_6 & 77 & 5 &\GR{A}{1}{+}\GR{A}{1} & 7 & $+$ 
\\[.5ex]  \hline
6a &  \GR{D}{4} & \vp_1{+}\vp_3 & 16 & 2 &\GR{G}{2} & 4 & $+$    \\
6b &         & m\vp_1{+}\vp_3, m{=}2{,}3 & 8(m{+}1) & &\GR{A}{4-m} & & $+$, if $m{=}3$ \\
6c &  & \tabcolsep=0.2em \begin{tabular}{c} $m\vp_1{+}\vp_3{+}\vp_4$ \\ $m{=}1,2$ 
\end{tabular}
& 8(m{+}2) & &\GR{A}{3-m} & & $+$ 
\\[.5ex]  \hline   
7a &  \GR{D}{5} & \vp_4 & 16 & 0 &\gt{so}_7{\ltimes}V_{\vp_3} & 3 & $+$    \\
7b &    & \vp_1{+}\vp_4 & 26 & 2 &\GR{B}{3} & 5 & $+$    \\
7c &    & 2\vp_1{+}\vp_4 & 36 & 5 &\GR{G}{2} & 7 & $+$    \\
7d &    & m\vp_1{+}\vp_4, m{=}3{,}4 & 16{+}10m & &\GR{A}{5-m} & & $+$   \\
7e &    & 2\vp_4 & 32 & 1 &\GR{G}{2} & 3 & $+$    \\
7f  &    & m\vp_1{+}2\vp_4, m{=}1{,}2 & 32{+}10m & &\GR{A}{3-m} & &  $+$, if $m{=}2$     \\
7g &    & 3\vp_4 \text{ or } 2\vp_4{+}\vp_5 & 48 & 6 &\GR{A}{1} & 7 &  $+$   \\[.4ex]
7h &    & \tabcolsep=0.2em \begin{tabular}{c} $m\vp_1{+}\vp_4{+}\vp_5$ \\ $0{\le}m{\le} 2$ 
\end{tabular} 
& 32{+}10m & 2{+}2m{+}m^2 &\GR{A}{3-m} & 5{+}m{+}m^2 & \tabcolsep=0.2em \begin{tabular}{c}
$-$, if $m{\le}1$ \\ $+$, if $m{=}2$ \end{tabular} 
\\[.5ex]  \hline
8a &  \GR{D}{6} & \tabcolsep=0.2em \begin{tabular}{c} $m\vp_1{+}\vp_5$\rule{0pt}{5mm} \\
$0{\le} m{\le} 4$ \end{tabular} 
& 32{+}12m & 1{+}m^2 &\GR{A}{5-m} & 6{-}m{+}m^2 & \tabcolsep=0.2em \begin{tabular}{c}
$+$, if $m=0{,}4$  \\  $-$, if $1{\le} m {\le} 3$ \end{tabular}  \\
8b &   & 2\vp_5 & 64 & 7 & 3\GR{A}{1} & 10 & $+$    \\
8c &   & \vp_5{+}\vp_6 & 64 & 4 & 2\GR{A}{1} & 6 & $+$     
\\[.5ex]  \hline
9a  & \GR{D}{7} & \vp_6 & 64 & 1 & 2\GR{G}{2} & 5 & $+$ \\    
9b  &   & m\vp_1{+}\vp_6, m{=}1{,}2 & 64{+}14m & & 2\GR{A}{3-m} & & $+$  \\[.4ex] 
\hline 
\end{tabular}
\end{center}
\end{table}

\textbullet \ \ Generic stabilisers for the representations in the tables are taken from~\cite{alela1}.
To verify that the generic isotropy groups are connected,  we use Proposition~4.10 and Remark~4.11 in \cite{gerry1}. In case of reducible representations, this can be combined with the group analogue of \cite[Lemma\,2]{alela1}.

\textbullet \ \ Apart from a generic isotropy group for $(G:V^*)$,   
we often have to compute the isotropy group $G_y$, where $y$ is a generic point of a $G$-stable divisor $D\subset V^*$, cf. Theorem~\ref{V-rank-1}.  
Mostly this is done by {\sl ad hoc} methods. Also the following observation is very helpful. 
Any divisor $D\subset V_1\oplus V_2$ projects dominantly to at least one factor $V_i$. Hence it
contains a subset of the form $\{x_i\}\times D_{i'}$, where 
$x_i\in V_i$ is generic, $D_{i'}\subset V_{i'}$ is a divisor,  and $\{i,i'\}=\{1,2\}$.  

\begin{table}[htbp]
\caption{The representations of the symplectic group with polynomial ring $\bbk[V]^G$ and non-trivial generic stabilisers}   \label{table-sp}
\begin{center}
\begin{tabular}{>{\sf}c<{}>{$}c<{$}>{$}c<{$}>{$}c<{$}>{$}c<{$}>{$}c<{$}>{$}c<{$}cc}
{\rus N0} &  G & V &  \dim V & \dim V\md G &  \h & \ind\es & ({\sf FA}) \\ \hline \hline
1 &  \GR{C}{n} &   \tabcolsep=0.2em \begin{tabular}{c}  $m\vp_1$,\\ $m {\le} 2n{-}1$  \end{tabular}
& 2mn & {m \choose 2}  &  
\tabcolsep=0.2em \begin{tabular}{ll}  
$\GR{C}{n-l}$,                                 & $m=2l$ \\
$\GR{C}{n-l}{\ltimes} \gt{heis}_{n-l}$, & $m=2l{-}1$  \end{tabular}
& {m \choose 2}{+}n{-}[\frac{m}{2}] & $+$  
\\[.6ex]  \hline
2 & \GR{C}{n} &  \vp_2              & 2n^2{-}n{-}1    & n-1 & n\GR{A}{1} & 2n{-}1 & $+$ \\
3 & \GR{C}{n} &  \vp_1{+}\vp_2 & 2n^2{+}n{-}1   & n-1 & \GR{U}{n} & 2n{-}1 & $+$ 
\\[.5ex]  \hline  
4  & \GR{C}{3} & \vp_3 & 14 & 1 & \GR{A}{2} & 3 & $+$ \\    
5  &   &  \vp_1{+}\vp_3 & 20 & 2 & \GR{A}{1} & 3 & $+$  \\
6  &   &  2\vp_2            & 28 & 8 & \te_1 & 9 & $+$
\\[.5ex] 
 \hline 
\end{tabular}
\end{center}
\end{table}

\textbullet \ \ Another major ingredient in obtaining the classification is (the presence of)
the ``codim--2'' property for $\es$. 
Some methods for checking the ``codim--2'' condition are presented in \cite[Sect.\,4]{md-ko}.
Similarly to the Ra\"is formula, see Eq.~\eqref{ind-sum}, we also have 
$$\dim\gt s_{\gamma{+}y} = \dim(\gt g_y)_{\bar\gamma}+(\dim V-\dim(G{\cdot}y)),$$
where $y\in V^*$, $\gamma\in\gt g$, and $\bar\gamma=\gamma|_{\gt g_y}$, cf. \cite[Eq.~(3${\cdot}$1)]{Y16}. Therefore, $\gt s$ has the ``codim--2" property if and only if 
\begin{itemize}
\item[({\sf i})] $\gt g_x$ with $x\in V^*$ generic  has the ``codim--2" property and 
\item[({\sf ii})] for any divisor $D\subset V$,  $\ind\gt g_y+(\dim V-\dim(G{\cdot}y))=\ind\gt s$ holds 
for all points $y$ of a non-empty  open subset $U\subset D$, cf. \cite[Eq.~(3${\cdot}$2)]{Y16}. 
\end{itemize}

\textbullet \ \ Finally, we recall an important class of semi-direct products. Let $\gt g=\gt g_0\oplus\gt g_1$ be 
a $\Z_2$-grading of $\g$, i.e., $(\g,\g_0)$ is a symmetric pair. Then the semi-direct product
$\es=\g_0{\ltimes}\g_1^{\sf ab}$, where $[\gt g_1^{\sf ab}, \gt g_1^{\sf ab}]=0$, 
is called the {\it $\Z_2$-contraction} of $\g$ related to the symmetric pair $(\g,\g_0)$. 
Set $l=\rk\g$ and let $H_1,\ldots,H_l$ be a set of the basic symmetric invariants of $\gt g$. 
Let $H_i^\bullet$ denote the bi-homogeneous component of $H_i$ that has the highest $\g_1$-degree. 
Then $H_i^\bullet$ is an $\es$-invariant in $\gS(\es)$~\cite{coadj}. We say that a $\Z_2$-contraction 
$\es=\g_0{\ltimes}\g_1^{\sf ab}$ is {\it good} if 
$\gS(\es)^{\es}$ is freely generated by the polynomials $H_1^\bullet, \ldots,H_l^\bullet$ for 
some well-chosen generators $\{H_i\}$. 
Note that $\deg H_i=\deg H_i^\bullet$ for the usual degree.

\section{An example in type ${\sf A}$}
\label{sect:1-A}

\noindent
The example considered in this section will be needed below in our treatment of $G=\SO_n$. 
It can also be regarded as a small step towards the classification in type {\sf A}.

Suppose 
that $G=\SL_n\subset \GL_n=\tilde G$ and $V=\bigwedge^2 \bbk^n {\oplus} (\bigwedge^2 \bbk^n)^*$. 
Then $\tilde{\gt s}=\tilde{\gt g}{\ltimes}V$ is the $\mathbb Z_2$-contraction of 
$\gt{so}_{2n}$ related to the symmetric pair $(\gt{so}_{2n},\gt{gl}_n)$. 
By~\cite[Theorem~4.5]{contr}, this $\mathbb Z_2$-contraction is good and satisfies {\it\bfseries KRC}.
Our goal is to describe $\gS(\gt s)^{\gt s}$ using the known description for $\tilde\es$.
Let us denote the basic symmetric invariants of $\tilde{\gt s}$ by $H_1,\ldots,H_{\ell}$, $F_1,\ldots, F_r$, where  $\deg F_i=2i$ and 
$\bbk[F_1,\dots,F_r]=\bbk[V^*]^{\GL_n}$. Then necessary  $\ell=[\frac{n+1}{2}]$, $r=[\frac{n}{2}]$.

\begin{prop}     \label{thm-1-A}
If $n=2r$ is even, then $\gS(\gt s)^{\gt s}$ is freely generated by 
$H_1,\ldots,H_r,F_1,\ldots,F_{r{-}1}$, $F_r',  F_{r{+}1}$ with $\deg F_r'=\deg F_{r+1}=r$. 
\end{prop}
\begin{proof}
Since $n$ is even, the generic isotropy group of the $\GL_n$-action on $V^*$ is 
$(\SL_2)^r$ and it lies in $\SL_n$. Therefore each $H_i$ lies in $\gS(\gt s)$, see Lemma~\ref{lm-sub}. 
The new generators $F_r'$, $F_{r{+}1}$ are the pfaffians on $\bigwedge^2 \bbk^n$ and  
$(\bigwedge^2 \bbk^n)^*$, respectively. 
We have 
\[
 \left(\sum\limits_{i=1}^{r} \deg H_i+\sum\limits_{j=1}^{r{-}1}\deg F_j\right)+2r=
 \frac{\dim\gt s+\ind\gt s}{2} .
 \] 
The generic isotropy groups of $(G{:}\bigwedge^2 \bbk^n)$ and $(\tilde G{:}\bigwedge^2 \bbk^n)$ are the 
same and $\tilde{\gt s}$ has the ``codim-2" property by \cite{coadj}. Therefore  $\es$ has the ``codim--2" property as well. The polynomials $F_1,\ldots,F_{r{-}1},F_r',F_{r{+}1}$ freely generate 
$\bbk[V^*]^G$~\cite{gerry1,ag79} and the other generators, $H_1,\ldots,H_r$,  are algebraically independent over $\bbk[V^*]$. Therefore  Theorem~\ref{thm-codim2} applies and provides the result. 
\end{proof}

The case of an odd $n$ is much more difficult, because a generic stabiliser for $(G{:}V)$  is not reductive. We conjecture that $\gS(\gt s)^{\gt s}$ is still a polynomial ring, but the proof would require a subtle detailed analysis of the generators $H_1,\ldots,H_\ell$. 
Since that case is not used in this paper, we postpone the exploration. Note only that 
if $n=3$, then there is an  isomorphism $\bigwedge^2 \bbk^3 \simeq (\bbk^3)^*$. 
The pair $(\SL_3,\bbk^3{\oplus}(\bbk^3)^*)$ was considered in \cite{Y}, where it is shown that 
the corresponding $\es$ has ({\sf FA}). 

\section{The classification for the orthogonal algebra}
\label{sect:so}

\noindent
In this section, $G=\Spin_n$. We classify the finite-dimensional rational representations $(G:V)$ 
such that $\mathsf{g.i.g.}(G:V)$ is infinite and 
the symmetric invariants of $\es=\g\ltimes V$ form a polynomial ring. The answer is given in Table~\ref{table:ort1}. 

\subsection{The negative cases in Table~\ref{table:ort1}}
Most of the negative cases (i.e., those having `$-$' in column ({\sf FA}) in Table~\ref{table:ort1}) are 
justified by Proposition~\ref{prop:trick} and the reductions of  Example~\ref{ex-ch-1}. Another similar
diagram is presented  below:   
\beq   \label{eq:chain2}
\raisebox{4ex}{\xymatrix{
(\Spin_{12}, 2\varphi_1+\varphi_5)\ar[dr] \\
(\Spin_{11}, \varphi_1+\varphi_5)\ar[r]
 &  \text{\framebox{$(\Spin_{10}, \varphi_4+\varphi_5)$} }. }}
\eeq
That is, our next step is to show that $(\Spin_{10}, \varphi_4+\varphi_5)$ 
does not have  ({\sf FA}). Once this is done, we will know that all the cases in Diagram~\eqref{eq:chain2} are indeed negative. 
Afterwards, only one negative case is left, namely  
$(\Spin_{12},\varphi_1+\varphi_5)$. 

\begin{thm}     \label{th-D5}
The semi-direct product  $\es=\gt{so}_{10}\ltimes(\varphi_4+\varphi_5)$ does not have  ({\sf FA}). 
\end{thm}
\begin{proof}
Here $G=\Spin_{10}$ is  a subgroup of $\Spin_{11}\subset \GL(V)$ and $V\simeq V^*$ as  a 
$\Spin_{11}$-module. A generic isotropy group in 
$\Spin_{11}$ is $\SL_5$. 
A generic isotropy group  in $\Spin_{10}$ is $\SL_4$. 
There is a divisor $D\subset V$ such that $G_y$ is connected and  $\gt g_y=\gt{sl}_3\ltimes\gt{heis}_3$ for 
a generic point $y\in D$. The stabiliser $\gt g_y$ is obtained  as an intersection of $\gt{sl}_5$ and a specially chosen  $\gt{so}_{10}\subset \gt{so}_{11}$. 

Assume that $\gt s$ has  ({\sf FA}).
Then $\gS(\gt s)^{\gt s}=\bbk[H_1,H_2,H_3,F_1,F_2]$, where $\bbk[V^*]^G=\bbk[F_1,F_2]$. According to 
Proposition~\ref{non-red}, the restrictions $H_i|_{\gt g+x}$ are generators of 
$\gS(\gt{sl}_4)^{\SL_4}$ for $x\in V^*$ generic. Therefore we may assume that $\deg_{\gt g}H_i=i{+}1$.
By Theorem~\ref{thm-dif},  there is $y\in D$ with $G_y$ as above  such that the differentials 
$\textsl{d}F_1,\textsl{d}F_2,\textsl{d}H_1,\textsl{d}H_2,\textsl{d}H_3$ are linearly independent on 
a non-empty open subset of $\gt g+y$ that is stable w.r.t. $G_y{\ltimes}\exp(V)$. 

Take $\xi=\gamma+y$ with $\gamma\in\gt g$ generic. Replacing $\gamma$ by another point in 
$\gamma+\ads(V)y$ we may safely assume that $\gamma$ is zero on $\Ann(\gt g_y)$. 
Let $\bar\gamma$ stand for the restriction of $\gamma$ to $\gt g_y$.
Then $\gt s_\xi=(\gt g_y)_{\bar\gamma}\oplus\bbk^2=(\gt t_2{\oplus}\bbk z)\oplus\bbk^2$, 
where $\bbk z$ is the centre of $\gt{heis}_3$, $\gt t_2$ is a Cartan subalgebra of $\gt{sl}_3$,
and $\bbk^2\subset V$.  

We have $(\textsl{d}F_i)_\xi\in \gt s_\xi\cap V=\bbk^2$. At the same time 
$(\textsl{d}H_i)_\xi=\eta_i+u_i$, where $u_i\in V$, $\eta_i\in \gt g$, and 
$\eta_i$ is the differential of $H_i|_{\gt g+y}$ at $\gamma$. Since $\gamma$ was chosen to be generic, 
the elements $\eta_1,\eta_2,\eta_3$ are linearly independent. Hence the restrictions 
${\bf h}_i:=H_i|_{\gt g+y}$ are algebraically independent. 

It can be easily seen that $\ind\gt g_y=3$ and that  
$\gt g_y$ satisfies  the ``codim--2" condition.  Since $\deg{\bf h}_i=i{+}1$, we have 
$\gS(\gt g_y)^{G_y}=\bbk[{\bf h}_1,{\bf h}_2,{\bf h}_3]$ by Theorem~\ref{thm-codim2}. 
But $z\in \gS(\gt g_y)^{G_y}$ and $\deg z=1$. A contradiction! 
\end{proof}

\begin{thm}    \label{th-D6}
The semi-direct product  $\es=\gt{so}_{12}\ltimes(\vp_1+\vp_5)$ does not have   ({\sf FA}). 
\end{thm}
\begin{proof}    Here $G=\Spin_{12}$ and
a generic isotropy group for the $G$-action on  $V_{\varphi_5}$ 
(resp. $V_{\varphi_1}\oplus V_{\varphi_5}$)  is $\SL_6$ (resp. $\SL_5$). 
Let  $\boldsymbol{f}$ be a $\Spin_{12}$-invariant quadratic form on $V_{\vp_1}\simeq V_{\vp_1}^*$.
Then $D=\{\boldsymbol{f}=0\}\times V_{\vp_5}^*$ is a $G$-stable divisor in $V^*$. It can be verified
that, for a generic point $y\in D$, one has 
$\g_y=\gt{sl}_4\ltimes \gt{heis}_4$ and $G_y$ is connected. 
 As in the proof of  Theorem~\ref{th-D5}, $\gS(\g_y)^{G_y}$ has an element of degree $1$, i.e., it 
is not generated by symmetric invariants of degrees $2,3,4,5$, but it would have been if $\es$ had  ({\sf FA}). 
\end{proof}

\subsection{The positive cases in Table~\ref{table:ort1}}
We now proceed to the positive cases. Note first that all the instances, where $\h$ is of type $\GR{A}{1}$,  
are covered by Example~\ref{ex-A1}. 

\begin{prop}[{\bf Item 1}]\label{item1}
The semi-direct product $\gt s=\gt{so}_n\ltimes m\bbk^n$ with $m < n$ has ({\sf FA}). 
\end{prop}
\begin{proof}
We have $G\lhd \tilde G$ with $\tilde G=\SO_n{\times}\SO_m$ and $\gt s\lhd\tilde{\gt s}$ for 
$\tilde{\gt s}=\tilde{\gt g}{\ltimes} V$. The Lie algebra $\tilde{\gt s}$ is the 
$\mathbb Z_2$-contraction of $\gt{so}_{n{+}m}$ related to the symmetric subalgebra 
$\gt{so}_n{\oplus}\gt{so}_m$. Let $x\in V^*$ be generic. Then $\tilde G_x=G_x=\SO_{n{-}m}$. 
According to \cite{coadj}, $\bbk[\tilde{\gt s}^*]^{\tilde{\gt s}}=\bbk[V^*]^{\tilde G}[H_1,\ldots,H_{\ell}]$ 
is a polynomial ring, $\ell=[\frac{n{-}m}{2}]$.    
By Lemma~\ref{lm-sub}, $H_i\in\gS(\gt s)$ for every $i$. Next, $\gt s$ has the ``codim--2" property if $m=1$ by \cite{coadj}, hence $\gt s$ always has it. The polynomials $H_i$ are algebraically independent 
over $\bbk(V^*)$ and $\bbk[V^*]^G$ has $\frac{m(m{+}1)}{2}$ generators of degree $2$.    
Thereby we have $\ind\gt s$ algebraically independent homogeneous invariants with the  
total sum of degrees being equal to  
$$
m(m{+}1)+\sum\limits_{i=1}^\ell \deg H_i=m(m{+}1)+ \boldsymbol{b}(\tilde{\gt s}) - m(m{+}1) =
\boldsymbol{b}(\tilde{\gt s})= 
\boldsymbol{b}(\gt{so}_{n{+}m}) = \boldsymbol{b}(\gt s). 
$$
According to Theorem~\ref{thm-codim2}, $\bbk[\gt s^*]^S=\bbk[V^*]^G[H_1,\ldots,H_\ell]$. 
\end{proof}

\begin{thm}[{\bf Item 9b}]      \label{th-D7}
The semi-direct product  $\es=\gt{so}_{14}\ltimes( \vp_1+\vp_6)$ has ({\sf FA}). 
\end{thm}
\begin{proof}  Here $G=\Spin_{14}$ and the pair $(\Spin_{14},V_{\varphi_6}^*)$ is of rank one.  Let $v\in V_{\varphi_6}^*$ be a generic
point. Then $G_v=L\times L$, where $L$ is the connected group of type $\GR{G}{2}$.
By Theorem~\ref{sur-1}, the restriction homomorphism 
$$
\psi_v: \bbk[\es^*]^S \to \bbk[\g^*\oplus V_{\vp_1}^* +v]^{G_v\ltimes\exp(V)} \simeq \bbk[\g_v^*\oplus V_{\vp_1}^*]^{G_v\ltimes \exp(V_{\vp_1})}
$$
is surjective. 
Further, $V_{\vp_1}\simeq\bbk^{14}=\bbk^7 \oplus \bbk^7$ as a $L\times L$-module, where each $\bbk^7$ is a simplest irreducible $\GR{G}{2}$-module.
Hence $G_v{\ltimes}\exp(V_{\varphi_1})=Q\times Q$, where  $Q=L\ltimes\exp(\bbk^7)$. 
The group $Q$  has a free algebra of symmetric invariants and $\ind\q=3$~\cite{exc}.   

There are irreducible tri-homogeneous polynomials $H_1,\ldots,H_6\in\bbk[\es^*]^S$ such that,
for a generic point $v\in V_{\varphi_6}^*$, their
images $h_i=\psi_v(H_i)$ generate $\gS(\q{\times}\q)^{Q{\times}Q}$.  Let $f$ be a basic $G$-invariant in $\bbk[V_{\varphi_6}^*]$. 

Although the group $G{\ltimes}\exp(\bbk^{14})$ is not reductive, we can argue in the spirit of 
\cite[Section~2]{Y16} and conclude that $\bbk[\es^*]^S[\frac{1}{f}]=\bbk[H_1,\dots,H_6,f,\frac{1}{f}]$. 
Then the equality
$$
    \bbk[\es^*]^S=\bbk[H_1,\dots,H_6,f]
$$
holds if and only if the restrictions of the polynomials $\{H_i\}$ to $V_{\varphi_1}^*\times D$ are 
algebraically independent, where $D=\{f=0\}\subset V_{\varphi_6}^*$. 

Let $G{\cdot}y\subset D$ be the dense open orbit. Then $G_y$ is connected and 
$\gt g_y=\gt l\ltimes\gt l^{\sf ab}$ is the Takiff Lie algebra in type $\GR{G}{2}$, $\gt l=\Lie L$.  
There is only one possible embedding of $\g_y$ into $\gt{so}_{14}$. Under the non-Abelian 
$\gt l$ the space $\bbk^{14}$ decomposes as a sum of two $7$-dimensional simple modules.  
The Abelian ideal $\gt l^{\sf ab}$ takes one copy of $\bbk^7$ into another. In other words, 
$\g_y\ltimes\bbk^{14}=\q \ltimes\q^{\sf ab}$.  
By~\cite[Example~4.1]{kot-Takiff}, Theorem~2.2 of the same paper \cite{kot-Takiff} applies to 
$\q$ and guarantees us that  
the symmetric invariants of $\q{\ltimes}\q^{\sf ab}$ form a polynomial ring in 
$6$ generators, where  the degrees of the basic invariants are the same as in the case of 
$\q\oplus\q$. 

It remains to observe that the proof of \cite[Theorem~2.3]{exc} can be repeated for the 
semi-direct product $(G{\ltimes}\exp(V_{\varphi_1})){\ltimes}\exp(V_{\varphi_6})$ 
producing a suitable modification of the elements $H_1,\ldots,H_6$, cf. Theorem~\ref{V-rank-1}. 
\end{proof}

\begin{cl}[{\bf Item 5a.}] The reduction
$$
 (\Spin_{14}, \varphi_1+\varphi_6) \longrightarrow
 (\Spin_{13}, \varphi_6)  
$$
shows that also $(\Spin_{13}, \varphi_6)$ has ({\sf FA}), see Proposition~\ref{prop:trick}.
\end{cl}

\begin{thm}\label{th-D7-2}
The semi-direct product  $\es=\gt{so}_{14}\ltimes(2\varphi_1+\varphi_6)$ has ({\sf FA}). 
\end{thm}
\begin{proof}  Here $G=\Spin_{14}$ and  
the proof follows the same lines as the proof of Theorem~\ref{th-D7}. We split the group 
$S$ as $(G{\ltimes}\exp(2V_{\varphi_1})){\ltimes}\exp(V_{\varphi_6})$. 
Now $Q=L{\ltimes}\exp(2\bbk^7)$ and again $G_v{\ltimes}\exp(2V_{\varphi_1})=Q\times Q$. 
By \cite{exc}, $\q$ has ({\sf FA}) and the ``codim--2" property. Here $\ind\q=4$ and we have 
eight polynomials $H_i\in\bbk[\es^*]^S$ such that their restrictions to  $\g\oplus(2V_{\vp_1}^*)+v$ 
generate $\gS(\q \oplus \q)^{Q\times Q}$.  These polynomials are
tri-homogeneous w.r.t. the decomposition $\es=\g \oplus 2V_{\vp_1} \oplus V_{\vp_6}$. 
Again $\g_v\ltimes(V_{\varphi_1}\oplus V_{\varphi_1})=\q\ltimes\q^{\sf ab}$, 
\cite[Theorem~2.2]{kot-Takiff} applies to $\q$ and assures that 
the symmetric invariants of $\q{\ltimes}\q^{\sf ab}$ form a polynomial ring in 
$8$ generators, where the degrees of the basic invariants are the same as in the case of 
$\q\oplus\q$. 
\end{proof}

\begin{cl}The reductions
$$
 (\Spin_{14}, 2\vp_1+\vp_6)\longrightarrow (\Spin_{13}, \vp_1+\vp_6+\bbk) \longrightarrow 
  (\Spin_{12}, \vp_5+\vp_6+\bbk)   
$$
show that the pairs $(\Spin_{13}, \vp_1+\vp_6)$ and  $(\Spin_{12}, \vp_5+\vp_6)$  also have 
({\sf FA}), see Proposition~\ref{prop:trick}.
\end{cl}

Many representations in types $\GR{D}{4}$, $\GR{B}{4}$, and $\GR{B}{3}$ are covered by reductions from 
$\GR{D}{5}$.  Among the type $\GR{D}{5}$  cases, the following one is easy to handle.  

\begin{ex}[{\bf Item7a}]
The pair $(\GR{D}{5}, \varphi_4)$ is of rank zero and therefore the open $\Spin_{10}$-orbit in $\bbk^{10}$ is big. 
The existence of the isomorphism $\bbk[\gt g+x]^{G_x{\ltimes}\exp(V)}\simeq\gS(\gt g_x)^{G_x}$ 
\cite{Y16} shows that  $\gS(\gt s)^{\gt s}\simeq \gS(\gt h)^{H}$, where $H$ 
is the isotropy group of an element in the open  orbit 
and $\h=\Lie H$. 
In order to be more explicit, $H$ is connected and $\gt h=\gt{so}_7{\ltimes}\bbk^8$, where $\gt{so}_7$ acts on $\bbk^8$ via 
the spin-representation. The algebra  $\gS(\gt h)^{\gt h}$ is 
free by \cite[Example~3.8]{Y16}.  
By a coincidence, the semi-direct product encoded by $(\GR{D}{5}, \varphi_4)$  is also a 
truncated maximal parabolic subalgebra $\gt p$ of $\GR{E}{6}$. The  symmetric invariants of 
$\gt p$ are studied in  \cite{FKs} and by a computer aided calculation it is shown there that 
$\cs(\gt p)^{\gt p}$ is a polynomial ring with three generators of degrees $6$, $8$, and $18$. 
\end{ex}

Below we list the `top' pairs that have to be treated individually. They are 
divided into two classes, in the first class $\dim V\md G=1$ and in the second  $\dim V\md G>1$. 
\beq \label{pairs}
\left\{\begin{array}{l}
\text{Rank one pairs:} \enskip
(\GR{B}{5},\varphi_5), (\GR{D}{5}, 2\varphi_4), (\GR{D}{6}, \varphi_5), (\GR{D}{7}, \varphi_6); \\
\text{higher rank pairs:} \enskip
 (\GR{D}{5}, \varphi_1+\varphi_4), (\GR{D}{5}, 2\varphi_1+\varphi_4), (\GR{D}{5}, 3\varphi_1+\varphi_4),
(\GR{D}{6}, 2\varphi_5).
\end{array}\right.
\eeq

\begin{thm} \label{rank-1-pairs}
The rank one pairs listed in \eqref{pairs} have ({\sf FA}).
\end{thm}
\begin{proof}
In case of $(\GR{D}{5}, 2\varphi_4)$ a generic stabiliser is of type $\GR{G}{2}$. This pair is covered by Lemma~\ref{l-rk1}. 
For the other three pairs, many conditions of Theorem~\ref{V-rank-1} are satisfied. For each pair, 
there is an open orbit $G{\cdot}y\subset D$, where $D$ stands for the zero set of the generator $F\in\bbk[V^*]^G$.  
It remains to inspect the symmetric invariants of $G_y$. 

A generic isotropy group for $(\GR{B}{5},\varphi_5)$ is $\SL_5$, $G_y$ is connected, and
 $\g_y$ is a $\mathbb Z_2$-contraction of 
$\gt{sl}_5$, the semi-direct product $\gt{so}_5\ltimes V_{\varphi_1^2}$, which is a good $\mathbb Z_2$-contraction \cite{coadj}. 

A generic isotropy group for $(\GR{D}{6}, \varphi_5)$ is $\SL_6$, 
 $G_y$ is connected, and
 $\gt g_y$ is a $\mathbb Z_2$-contraction of 
$\gt{sl}_6$, the semi-direct product $\gt{sp}_6\ltimes V_{\varphi_2}$, which is a good $\mathbb Z_2$-contraction \cite[Theorem~4.5]{contr}. 

A generic isotropy group for $(\GR{D}{7}, \varphi_6)$ is $L\times L$, where $L$ is the connected group  of 
type $\GR{G}{2}$, $G_y$ is connected, and $\gt g_y$ is the Takiff algebra 
$\gt l{\ltimes}\gt l^{\sf ab}$, where $\gt l=\Lie L$.  The basic symmetric invariants of  $\gt g_y$ have the 
same degrees as in the case of  $\gt l\oplus\gt l$ \cite{takiff}. 
\end{proof}

\begin{ex}[{\bf Item 7d}]    \label{ex:7d}
For the pair  $(\GR{D}{5}, 3\varphi_1+\varphi_4)$, a  generic isotropy group is connected and is of type  
$\GR{A}{2}$.  Let $D\subset V^*$ be a $G$-invariant divisor. Then there are at least two copies of 
$\bbk^{10}$ in $V^*$ such that the projection of $D$ on each of them is surjective. For a generic $y\in D$, 
$G_y=(\Spin_8)_{\tilde y}$, where $\tilde y$ is a generic point of a $\Spin_8$-invariant divisor 
$\tilde D\subset V_{\varphi_1}\oplus V_{\varphi_3} \oplus V_{\varphi_4}$ (here highest weights of $\Spin_8$ are meant). Continuing the computation one obtains that 
$G_y=L_v$, where $L$ is the connected group of type $\GR{G}{2}$ and $v$ is a highest weight vector in 
$\bbk^7$.  The group $L_v$ has a free algebra of symmetric invariants generated in degrees $2$ and 
$3$, see \cite[Lemma~3.9]{exc}. 
Therefore Lemma~\ref{l-A2} applies. 
\end{ex}

The remaining  three  higher rank pairs listed in \eqref{pairs} require elaborate arguments. 
For all of them, Theorem~\ref{sur-1} will be the starting point. 
Note that the pair $(\SO_n,\bbk^n)$ is of rank one. We let 
$(\,.\,,.\,)$ denote a non-degenerate $\SO_n$-invariant scalar product on $\bbk^n$.

\begin{thm}[{\bf Item 7b.}]        \label{thm-7b} 
The semi-direct product   $\es=\gt{so}_{10}\ltimes(\varphi_1+\varphi_4)$ has ({\sf FA}).
\end{thm}
\begin{proof}
Here $G=\Spin_{10}$ and we use the reduction 
\beq   \label{eq:p-chain1}
(\Spin_{10}, \varphi_1{+}\varphi_4) \to (\Spin_{9},\varphi_4) 
\eeq
in the increasing direction, starting from the smaller representation and its invariants. 
By Theorem~\ref{sur-1},  the restriction homomorphism 
$$
\psi_v\!: \bbk[\gt s^*]^{\gt s} \to \bbk[\gt g^*\oplus V_{\varphi_4}^* +v]^{G_v\ltimes\exp(V)} \simeq \bbk[\gt g_v^*\oplus V_{\varphi_4}^*]^{G_v\ltimes\exp(V_{\varphi_4})}
$$
is surjective for generic $v\in V_{\varphi_1}^*$. Here $G_v=\Spin_9$.  The group $Q=G_v\ltimes\exp(V_{\varphi_4})$ has a free algebra of symmetric invariants \cite[Theorem~4.7]{coadj}.  More explicitly, $\gS(\q)^Q$ is generated 
by $(\,.\,,.\,)$ on $\bbk^{16}$ and three bi-homogeneous polynomials ${\bf h}_1$, ${\bf h}_2$, ${\bf h}_3$
of bi-degrees $(2,4)$, $(4,4)$, $(6,6)$. Note that each generator is unique up to a non-zero scalar. 
Whenever $(\xi,\xi)\ne 0$ for $\xi\in V_{\varphi_4}^*$, we have 
${{\bf h}_i}|_{\gt{so}_9{+}\xi}=\Delta_{2i}$, where each $\Delta_{2i}$ is a basic symmetric invariant of 
$\gt{so}_7=(\gt{so}_9)_\xi$.  The generators $\Delta_{2i}$ are now fixed and they do not depend on 
the choice of $\xi$. 

Take $H_i\in \bbk[\gt s^*]^{\gt s}$ with $\psi_v(H_i)={\bf h}_i$. Without loss of generality, we may assume that 
$H_i$ is homogeneous w.r.t. to $\gt g$ and $V_{\varphi_4}$. The uniqueness of  
the basic symmetric $\q$-invariants, allows us to take a suitable tri-homogeneous component of 
each $H_i$, see Theorem~\ref{sur-1}. 
Now assume that each $H_i$ is irreducible. 
Whenever $(\xi,\xi)\ne 0$ for $\xi\in \bbk^{16}$ and $(\eta,\eta)\ne 0$ for $\eta\in\bbk^{10}$, we have 
${H_i}|_{\gt g{+}x}=a_x \Delta_{2i}$, where $x=\eta+\xi$ and $a_x\in\bbk^{\!^\times}\!$. 
 
According to \cite[Lemma~3.5({\sf ii})]{Y16}, 
we have $\gS(\gt s)^{\gt s}=\bbk[V^*]^G[H_1,H_2,H_3]$ if and only if 
the restrictions $H_i|_{\gt g{\times}D}$ are algebraically independent 
over $\bbk[D]^G$ for each $G$-invariant divisor $D\subset V^*$. 

If $D$ contains a point $av+\xi$ with $\xi\in\bbk^{16}$ and $a\ne 0$, a relation among 
$H_i|_{\gt g{\times}D}$ leads to a relation among the restrictions of ${\bf h}_i$  to $\gt{so}_9{\times}\tilde D$ 
for some $\Spin_9$-invariant divisor $\tilde D\subset \bbk^{16}$. Moreover, this new relation is  over $\bbk[\{v\}{\times}D]^{G_v}=\bbk$. Since the polynomials ${\bf h}_i$ freely generate $\gS(\q)^{\q}$ over 
$\bbk[V_{\varphi_4}^*]^{\Spin_9}$, nothing of this sort can happen. 
Therefore there is a unique  suspicious divisor, namely, the divisor $D=\tilde D\times \bbk^{16}$, where 
$\tilde D=\{u\in\bbk^{10} \mid (u,u)=0\}$. 

Since each $H_i$ is irreducible, it is non-zero on $\gt g{\times} D$. 
Therefore there is  a point $\xi\in\bbk^{16}$ such that $(\xi,\xi)\ne 0$ and ${H_i}|_{\gt g{\times}\tilde D{\times}\{\xi\}}\ne 0$ for all $i$. Here $G_\xi=\Spin_7{\ltimes}\exp(\bbk^8)$ and $\bbk^{10}\subset V^*$ decomposes 
as $\bbk\oplus\bbk^8\oplus\bbk$ under $G_\xi$. The Abelian ideal $\bbk^8$ of $\gt g_\xi$ takes 
$\bbk$ to $\bbk^8$ and then $\bbk^8$ to another copy of $\bbk$. 
Note that the vectors in each copy of $\bbk$ are isotropic.  
Take $u\ne 0$ in the first copy and $u'\ne 0$ in the second copy of $\bbk$. 
Set $\eta_t=u{+}tu'$, $x_t=\eta_t{+}\xi$ for $t\in\bbk$, $y=u{+}\xi$. 
Then $G_{x_t}=G_y\simeq \Spin_7$.

We have $(\eta_t,\eta_t)\ne 0$ for $t\ne 0$ and 
hence 
${H_i}|_{\gt g+x_t}=a_t \Delta_{2i}\ne 0$, whenever $t\ne 0$. Here 
$a_t \Delta_{2i}\in\gS(\gt g_{x_t})=\gS(\gt g_y)$. 
Clearly ${H_i}|_{\gt g{\times}\{y\}}=\lim\limits_{t \to 0} a_t\Delta_{2i}$ and it is either zero or a non-zero scalar multiple of $\Delta_{2i}$. If the second possibility takes place for all $i$, when the restrictions of $H_i$ to 
$\gt g\times D$ are algebraically independent over $\bbk[D]$.   Thus, it remains to prove
that ${H_i}|_{\gt g{\times}\{y\}}\ne 0$ for all $i$. 

Assume that  ${H_i}|_{\gt g{\times}\{y\}}=0$.  Then $H_i$ vanishes on 
$\gt g{\times}G_\xi{\cdot}u{\times}\{\xi\}$ and also on $\gt g{\times}G_\xi{\cdot}\bbk u{\times}\{\xi\}$, 
since $H_i$ is tri-homogeneous.   
The subset  $G_\xi{\cdot}\bbk u$ is dense in $\tilde D$ (it equals $\tilde D\setminus\{0\}$), hence 
$H_i$ vanishes on $\gt g{\times}\tilde D{\times}\{\xi\}$, too. However, this contradicts the choice of $\xi$.  
\end{proof}

\begin{thm}[{\bf Item 7c.}]     \label{thm-7c} 
If $\es$ is given by $(\GR{D}{5}, 2\varphi_1+\varphi_4)$, then 
$\bbk[\es^*]^S=\bbk[V^*]^G[H_2,H_6]$ is a polynomial ring 
and the multi-degrees of $H_i$ are $(2,2,2,4)$, $(6,4,4,8)$. 
\end{thm} 
\begin{proof}
For this pair, the chain of reductions is 
\beq   \label{eq:p-chain2}
(\Spin_{10}, 2\varphi_1{+}\varphi_4) \to (\Spin_{9}, \varphi_1{+}\varphi_4{+}\bbk) \to  
(\Spin_8, \varphi_3{+}\varphi_4{+}\bbk) \to (\Spin_{7}, \varphi_3{+}\bbk)
\eeq
and again we are tracing the chain from the smaller groups to the larger. 

By~\cite[Prop.~3.10]{Y16},  the symmetric invariants  of $\Spin_7\ltimes\exp(V_{\vp_3})$ are freely generated by the following three polynomials:  the scalar product $(\,.\,,.\,)$ on $V_{\vp_3}\simeq\bbk^8$, ${\bf h}_2$, and ${\bf h}_6$. Here  
the bi-degrees the last two are $(2,2)$, $(6,4)$.  
We are lucky that all three generators are unique (up to a scalar) and 
$\gt{so}_7\ltimes V_{\vp_3}$ has the ``codim--2" property.  
One can easily deduce that all items in~\eqref{eq:p-chain2} have the ``codim--2" property.  
A generic isotropy group for $(\Spin_7 : V_{\vp_3})$, say $L$, is the connected simple group of type 
$\GR{G}{2}$.
Take $u\in V_{\vp_3}$ with  $(u,u)\ne 0$. Then $(\gt{so}_7)_u=\gt l=\Lie L$.   
Let $h_2, h_6\in\gS((\gt{so}_7)_u)$ be  the restrictions of ${\bf h}_2$, ${\bf h}_6$ to 
$\gt{so}_7{+}u$. Then $h_2$ and $h_6$ generate $\gS(\gt l)^L$.  We have $\dim\gS^2(\gt l)^L=1$, 
the generator of degree $2$ is unique (up to a non-zero scalar). 
In the space $\gS^6(\gt l)^L=\bbk h_2^3\oplus\bbk h_6$, the generator $h_6$ is characterised by the property that it is 
the restriction of an invariant of  $\Spin_7\ltimes\exp(V_{\vp_3})$ of bi-degree $(6,4)$. 
This property does not depend on the choice of $u$. 

Consider next $\gt s_2:=\gt{so}_8{\ltimes}(V_1{\oplus}V_2)$, where $V_1=V_{\varphi_3}$, $V_2=V_{\varphi_4}$. Choose $v\in V_1^*$ with $(v,v)\ne 0$. By 
Theorem~\ref{sur-1}, there are $\hat h_2, \hat h_6\in\gS(\gt s_2)^{\gt s_2}$ such that 
${\hat h_i}|_{\gt{so}^8{\oplus}V_2^*{+}v}={\bf h}_i$. One can safely replace $\hat h_2$
by its component of degrees $2$ in $\gt{so}_8$, $2$ in $V_2$ and replace 
$\hat h_6$ by  its component of degrees $6$ in $\gt{so}_8$, $4$ in $V_1$. 
The uniqueness of generators in the case of $\Spin_7\ltimes\exp(V_{\vp_3})$ allows also to take 
tri-homogeneous components. Suppose now that each ${\hat h}_i$ is irreducible. 
Set $a_i=\deg_{V_2}\hat h_i$. Choose $v_2\in V_2^*$ with $(v_2,v_2)\ne 0$. 
The restriction ${\hat h_2}|_{\gt{so}_8{\oplus}V_2{+}v_2}$ is an invariant of 
bi-degree $(2,a_2)$ and either $a_2=2$ or this restriction is  divisible by the invariant of 
bi-degree $(0,2)$. In the last case, ${\hat h}_2$ is divisible by a generator of
$\bbk[V_2]^{\SO_8}$. A contradiction. 
Since ${\hat h_6}\vert_{\gt{so}_8{+}v{+}v_2}=h_6$ and since in addition $\hat h_6$ is irreducible, 
the restriction ${\hat h_6}|_{\gt{so}_8{\oplus}V_1^*{+}v_2}$ is an invariant of bi-degree $(6,4)$, i.e.,
$a_6=4$. Making use of Theorem~\ref{thm-codim2}, we conclude that 
$\bbk[\gt s_2^*]^{\gt s_2}=\bbk[V_1^*{\oplus}V_2^*]^{\Spin_8}[\hat h_2,\hat h_6]$. 
 
The $\Spin_9$-actions on $V_{\varphi_1}=\bbk^9$ and $V_{\varphi_4}=\bbk^{16}$ are of rank one.  
By~\cite{alela1}, $\mathsf{g.i.g.}(\Spin_9{:}V_{\vp_4})=\Spin_7$, and $\bbk^9\vert_{\Spin_7}$ is 
the $\Spin_7$-module $V_{\varphi_3}\oplus\bbk$.  
The restriction homomorphism $\bbk[V_{\varphi_1}^*\oplus V_{\varphi_4}^*]^{\Spin_9}\to \bbk[V_{\varphi_3}^*\oplus\bbk]^{\Spin_7}$ is onto. 
Using Theorem~\ref{sur-1} and the reductions 
\beq   \label{eq:red1}
\raisebox{4.2ex}{\xymatrix@R-8mm@C+7mm{
& (\Spin_8,\varphi_3{+}\varphi_4)  \ar[dr] & \\
(\Spin_{9}, \varphi_1{+}\varphi_4) \ar[ur] \ar[dr] &  & (L,0) \\
 & (\Spin_{7}, \varphi_3{+}\bbk)\ar[ur] & }}
\eeq
we prove that there are 
algebraically independent over $\bbk[V_{\varphi_1}^*{\oplus}V_{\varphi_4}^*]$ 
symmetric invariants of tri-degrees $(2,2,4)$, $(6,4,8)$ w.r.t. $\gt{so}_9\oplus\bbk^9\oplus\bbk^{16}$. 
They generate the ring of symmetric invariants related to $(\Spin_9,\varphi_1{+}\varphi_4)$ over  $\bbk[V_{\varphi_1}^*{\oplus}V_{\varphi_9}^*]^{\Spin_9}$ by Theorem~\ref{thm-codim2}.  

One can make a reduction step from $\es$ to 
$(\Spin_9, V_{\varphi_1}{\oplus}V_{\varphi_9})$ using either of 
the two copies of  $V_{\varphi_1}$. This allows one to find 
algebraically independent over $\bbk[V^*]$ polynomials $H_2,H_6\in\bbk[\gt s^*]^{\gt s}$ of 
multi-degrees  $(2,2,2,4)$ and $(6,4,4,8)$, respectively.  The basic invariants on $V^*$ are 
of degrees $2,2,2,3,3$. Thus, the total sum of degrees is 
\[
10+22+12=44 \ \text{ and } \ \dim\gt s+\ind\gt s=45+20+16+7=88. 
\]
Therefore, by Theorem~\ref{thm-codim2}, we have $\bbk[\gt s^*]^S=\bbk[V^*]^G[H_2,H_6]$. 
\end{proof}

The case of $(\GR{D}{6},2\varphi_5)$ is very complicated. 
We begin by introducing some notation and stating a few facts related to this pair.  
First, $V_{\varphi_5}\simeq V_{\varphi_5}^*$ as a $G$-module. 
Second, the representation of $G$ on $V_{\varphi_5}$ is of rank one  
and $\bbk[V_{\varphi_5}^*]^G=\bbk[F]$, where $F$ is a homogeneous polynomial of degree $4$.  
It would be convenient to write $V=V_1\oplus V_2$, where each $V_i$  is a copy   of 
$V_{\varphi_5}$  and let $F$ stand for the generator of 
$\bbk[V_1^*]^G$. Further, there is a natural action of $\SL_2$ on $V$. 
We suppose that $\left(\begin{array}{cc} 0 & 0 \\ 1 & 0 \end{array}\right){\cdot}V_2=0$ and that 
$\left(\begin{array}{cc} 0 & 0 \\ 1 & 0 \end{array}\right){\cdot}V_1=V_2$ for 
$\left(\begin{array}{cc} 0 & 0 \\ 1 & 0 \end{array}\right)\in\gt{sl}_2$. 
The ring $\bbk[V^*]^G$ has $7$ generators:
$$
F=F_{(4,0)}, F_{(3,1)}, F_{(2,2)}, F_{(1,3)}, F_{(0,4)}, F_{(1,1)}, F_{(3,3)}. 
$$
Here  $F_{(\alpha,\beta)}$ stands for  a particular $G$-invariant
in $\gS^{\alpha}(V_1)\gS^{\beta}(V_2)$. It is assumed that 
the first  five polynomials  build an irreducible $\SL_2$-module and 
that the last two are $\SL_2$-invariants. 
 
We let $\SL_2$ act on $\gt g$ trivially and thus obtain an action of $\SL_2$ on 
$\bbk[\gt s^*]^S$. Note that 
if $H\in\bbk[\gt s^*]$ and $\deg_{V_1} H>\deg_{V_2} H$, then 
$\left(\begin{array}{cc} 0 & 0 \\ 1 & 0 \end{array}\right){\cdot}H\ne 0$. 

Let $v\in V_1^*$ be a generic point and 
$$
\psi_v\!: \bbk[\gt s^*]^S \to \bbk[\gt g{\oplus}V_2^*{+}v]^{G_v{\ltimes}\exp(V)} \simeq \gS(\gt g_v{\ltimes}V_2)^{G_v{\ltimes}\exp(V_2)}
$$    
be the corresponding restriction homomorphism. 
Here $G_v=\SL_6$   and $$V_2=\bigwedge^2 \bbk^6 \oplus (\bigwedge^2 \bbk^6 )^* \oplus 2\bbk$$ as a $G_v$-module. 
Set $\q=\gt g_v{\ltimes}V_2$. 
By Proposition~\ref{thm-1-A}, 
$\gS(\q)^{\q}$ is a polynomial ring and $\bbk[\q^*]^{\q}=\bbk[V_2^*]^{\SL_6}[{\bf h}_1,{\bf h}_2,{\bf h}_3]$, where 
the generators ${\bf h}_i$ are of bi-degrees $(2,4)$, $(2,6)$, $(2,8)$. 

Let $N_{G}(\bbk v)$ be the normaliser of the line $\bbk v$. 
Then $N_{G}(\bbk v)=C_4{\times}G_v$, where $C_4=\left<\zeta\right>$ is a cyclic group of order $4$. 
It is not difficult to see that $\Ad(\zeta)A=-A^t$ for each $A\in\gt g_v$
and that $\zeta{\cdot}{\bf h}_k=(-1)^k{\bf h}_k$ for each $k\in\{1,2,3\}$. 
The element $\zeta^2$ multiplies $\psi_v(F_{(1,3)})$ and $\psi_v(F_{(1,1)})$ by $-1$, 
the product $\psi_v(F_{(1,3)})\psi_v(F_{(1,1)})$ is a $C_4$-invariant.  

There is a subgroup $C_4\subset N_{G}(\bbk v)\times \GL(V_1^*)$, which stabilisers $v$. 
This means that if $H\in\bbk[\gt s^*]^S$ is homogeneous in $V_1$, then 
$\psi_v(H)$ is an eigenvector of $C_4\subset N_G(\bbk v)$ and 
the corresponding eigenvalue depends only on $\deg_{V_1}H$.

\begin{thm}[{\bf Item 8b}]    \label{th-D6-25} 
If $\gt s$ is given by the pair $(\GR{D}{6}, 2\varphi_5)$, then $\bbk[\gt s^*]^S=\bbk[V^*]^G[H_1,H_2,H_3]$ 
is a polynomial ring and the tri-degrees of $H_i$ are $(2,4,4)$, $(2,6,6)$, $(2,8,8)$.  
\end{thm}
\begin{proof}
According to Theorem~\ref{sur-1}, there are homogeneous  in $V_1$ elements  
$H_1,H_2,H_3\in\gS(\gt s)^{\gt s}$ such that 
$\psi_v(H_i)={\bf h}_i$. There is no harm in assuming that these polynomials are
tri-homo\-ge\-neous. 
Suppose that $b_i=\deg_{V_1} H_i$ is the minimal possible. 
Set $a_i=\deg_{V_2} H_i$. Then $a_1=4$, $a_2=6$,  $a_3=8$.
The eigenvalues of $\zeta$ on ${\bf h}_i$ indicate that   
$a_i\equiv b_i \!\pmod{4}$ for each $i$. 

Suppose for the moment that $a_i=b_i$ for all $i$. 
It is not difficult to see that $\gt s$ satisfies the 
``codim--2" condition. The elements ${\bf h}_1, {\bf h}_2, {\bf h}_3$ are algebraically independent
over $\bbk(V_2)$, hence  $H_1, H_2, H_3$ are algebraically independent 
over $\bbk(V^*)$. Thus, we have ten algebraically independent 
homogeneous invariants. The  total sum of their degrees is 
$$
2+6+20+10+14+18=70    \ \text{ and } \dim\gt s+\ind\gt s=66+64+10=140. 
$$
Thereby $\bbk[\gt s^*]^S=\bbk[V^*]^G[H_1,H_2,H_3]$ by Theorem~\ref{thm-codim2}.
It remains to show that the assumption is correct. 

For a generic $v'\in V_2^*$, 
$\gt g_{v'}{\ltimes}V_1\simeq\q$ 
and each ${H_i}|_{\gt g{\oplus}V_1^*{+}v'}$ is a symmetric 
invariant of $\gt g_{v'}{\ltimes}V_1$ of degree $2$ in $\gt g_{v'}$.
Since the restrictions ${H_i}|_{\gt g{+}v{+}v'}$ are the basic symmetric invariants 
of $G_{v{+}v'}=(\SL_2)^3$,  
the restrictions of $H_i$ to $\gt g{\oplus}V_1^*{+}v'$ 
are algebraically independent over $\bbk[V_1^*]$. Thereby $\sum b_i\ge 18$ and 
$b_i\ge 4$ for each $i$. Moreover, if $b_1=4$, then $b_2\ge 6$. 

Set $\tilde H_i:=\left(\begin{array}{cc} 0 & 0 \\ 1 & 0 \\ \end{array}\right){\cdot}H_i$. 
If $b_i>a_i$, then $\tilde H_i\ne 0$. 
We have $\psi_v(\tilde H_i)\in\gS(\q)^{\q}$ and $\deg_{\gt g_v}\psi_v(\tilde H_i)=2$. 
Therefore $\psi_v(\tilde H_i)$ is a linear combination of ${\bf h}_j$ with 
coefficients from $\bbk[V_2^*]^{\SL_6}$. Moreover, each coefficient is an 
eigenvector of $\zeta$. The first element, $H_1$, can be handled easily. 

Assume that $\tilde H_1\ne 0$.  Then $\psi_v(\tilde H_1)={\bf f}\,{\bf h}_1$ with non-zero 
${\bf f}\in V_2^{G_v}$ and  this 
${\bf f}$ is an eigenvector of $\zeta$.  Since $\deg_{V_1} \tilde H_1 \equiv 3 \! \pmod{4}$, 
 ${\bf f}=\psi_v(F_{(3,1)})$ (up to a non-zero scalar).  Since
$\psi_{av}(\frac{\tilde H_1}{F_{3,1}})=a^{b_1{-}4}{\bf h}_1$ for each $a\in\bbk^{\!^\times}\!$
and since $F_{(3,1)}$ and $F$ are coprime, we have 
$\frac{\tilde H_1}{F_{3,1}}
\in\bbk[\gt s^*]^S$.  Also  
$\psi_v(\frac{\tilde H_1}{F_{3,1}})={\bf h}_1$. Clearly 
$\deg_{V_1} \frac{\tilde H_1}{F_{3,1}} = b_1{-}4 < b_1$. A contradiction with the choice of $H_1$.  
We have established that $\psi_v(\tilde H_1)=0$. Hence $b_1=4$ and $H_1$ is an $\SL_2$-invariant.   

Certain further precautions are needed. It may happen that $H_2$ (or $H_3$) does not lie 
in a simple $\SL_2$-module. In that case we replace $H_2$ (or $H_3$) by a suitable 
(and suitably normalised)  component of the same tri-degree, which lies in a simple $\SL_2$-module and 
which restricts to ${\bf h}_2+p$ with $p\in \gS^2(V_2){\bf h}_1$ 
(or to ${\bf h}_3+p$  with $p\in \gS^4(V_2){\bf h}_1{\oplus}\gS^2(V_2){\bf h}_2$) 
on $\gt g{\oplus}V_2^*{+}v$. One may say that ${\bf h}_2$ was (or ${\bf h}_2$ and ${\bf h}_3$ were) 
 changed as well, so that the conditions $\psi_v(H_i)={\bf h}_i$ are not violated. 
We also normalise $F$ in such a way  that  $F(v)=1$. 
Some other normalisations are done below without mentioning. 

Assume that $\tilde H_2\ne 0$ and that 
$\psi_v(\tilde H_2)\in \gS^3(V_1){\bf h}_1$. 
Then $\tilde H_2\in \bbk(V^*) H_1$ and so does $H_2$, which is equal  to 
$\left(\begin{array}{cc} 0 & 1 \\ 0 & 0 \\ \end{array}\right){\cdot}\tilde H_2$ up to a non-zero scalar. 
A contradiction, here $\psi_v(H_2)\ne {\bf h}_2$. 
Knowing that $H_2$ is an $\SL_2$-invariant, we can use a similar argument in order to prove that   
$\psi_v(\tilde H_3) \not \in \gS^5(V_1){\bf h}_1{\oplus}\gS^3(V_1){\bf h}_2$ in case $\tilde H_3\ne 0$. 

We will see below that 
if $b_i>a_i$, then $\tilde H_i={\bf H}_i+\frac{F_{(3,1)}H_i}{F}$, where 
${\bf H}_2\in \bbk(V^*)^G H_1$  and 
${\bf H}_3\in\bbk(V^*)^G H_1 \oplus \bbk(V^*)^G H_2$. Recall that $F$ and $F_{(3,1)}$ are
coprime.  In case ${\bf H}_i\in\bbk[\gt s^*]$, 
we can replace $H_i$ with $\frac{H_i}{F}\in\bbk[\gt s^*]$ decreasing $\deg_{V_1} H_i$  by $4$. 
The main difficulties lie with non-regular ${\bf H}_i$.

{\bf Modification for $H_2$.}  Assume that $b_2>6$. 
Then $\psi_v(\tilde H_2)={\bf f}_{3}{\bf h}_1+{\bf f}_{(3,1)}{\bf h}_2$ 
with ${\bf f}_{3}\in\gS^3(V_2)^{G_v}$, 
${\bf f}_{(3,1)}\in V_2^{G_v}$ and ${\bf f}_{(3,1)}\ne 0$. 
Both coefficients are eigenvectors of $\zeta$. We have 
${\bf f}_{(3,1)}=\psi_v(F_{(3,1)})$ and ${\bf f}_3$ is the image of 
$$c_1 F_{(1,3)}+ F'_{(5,3)}+ c_2 F_{(3,1)}^3,$$
where $c_1,c_2\in\bbk$ and $F'_{(5,3)}$ is some $G$-invariant 
in $\gS^5(V_1)\gS^3(V_2)$. 
Set $\delta:=\frac{b_2{-}6}{4}$ and
$$
{\bf H}_2:=(c_1 F^\delta F_{(1,3)} + F^{\delta{-}1} F'_{(5,3)}  +c_2 F^{\delta{-}2} F_{(3,1)}^3)  H_1.
$$ 
Then $\psi_{av}(\tilde H_2-{\bf H}_2)=a^{b_2{-}1}{\bf f}_{(3,1)}{\bf h}_2$ for all $a\in\bbk^{\!^\times}\!$. 
If ${\bf H}_2\not\in\bbk[\gt s^*]$, then 
$\delta=1$ and $c_2\ne 0$.  
Here 
$$
\tilde H_2-c_1 F F_{(1,3)} H_1 - F'_{(5,3)} H_1 - c_2\frac{F_{(3,1)}^3H_1}{F}=\frac{ F_{(3,1)} H_2}{F}
$$
and 
$$
\frac{  F_{(3,1)} H_2 + c_2 F_{(3,1)}^3 H_1}{F} \in \bbk[\gt s^*]. 
$$
Since $F$ and $F_{(3,1)}$ are  coprime, we have 
$$
\hat H_2=\frac{  H_2 + c_2 F_{(3,1)}^2 H_1}{F} \in \bbk[\gt s^*]. 
$$
In this case we replace ${\bf h}_2$ with ${\bf h}_2+c_2 {\bf f}_{(3,1)}^2 {\bf h}_1$ and 
$H_2$ with $\hat H_2$. This does not violate the property $\zeta^2{\cdot}{\bf h}_2=-{\bf h}_2$. 
Now $\deg_{V_1}H_2=\deg_{V_2} H_2=6$. 
If $\left(\begin{array}{cc} 0 & 1 \\ 0 & 0 \\ \end{array}\right){\cdot}H_2\ne 0$, then this is an invariant 
of tri-degree $(2,7,5)$ and hence lies in $\bbk(V^*)H_1$. But then also $H_2\in \bbk(V^*)H_1$.
This new contradiction shows that $H_2$ is an $\SL_2$-invariant.

{\bf Modification for $H_3$.} 
Now we know that $b_2=6$ and therefore $b_3\ge 8$. 
Assume that $b_3>8$. 
Then
$$
\psi_v(\tilde H_3)={\bf f}'_5{\bf h}_1+{\bf f}'_{3}{\bf h}_2+{\bf f}_{(3,1)}{\bf h}_3
$$ 
with 
${\bf f}'_k\in\gS^k(V_2)^{G_v}$, 
${\bf f}_{(3,1)}\in V_2^{G_v}$. 
All three coefficients are  eigenvectors of $\zeta$. 
Studying the eigenvalues one concludes that ${\bf f}_{(3,1)}=\psi_v(F_{(3,1)})$, 
${\bf f}'_3$ is the image of  
$s_1F_{(1,3)}+F'_{(5,3)} +s_2 F_{(3,1)}^3$, where 
$F'_{(5,3)}\in\gS^5(V_1)\gS^3(V_2)$, $s_i\in\bbk$,  
and finally ${\bf f}'_5$ is the image of a rather complicated expression
$\sum\limits_{j=0}^3  F'_{(4j{+}3,5)}$.  
Set $\nu:=\frac{b_3{-}8}{4}$
and 
$$
{\bf H}_3:=(\sum_{j=0}^{3}  F'_{(4j{+}3,5)} F^{\nu-j} ) H_1 + (s_1 F_{(1,3)} F^\nu + F'_{(5,3)} F^{\nu{-}1} + s_2 F_{(3,1)}^3 F^{\nu{-}2}) H_2.  
$$ 
As above, $\tilde H_3-{\bf H}_3=\frac{F_{(3,1)}H_3}{F}$. 
If ${\bf H}_3\not\in \bbk[\gt s^*]$, then $\nu=2$ or $\nu=1$. 

Suppose that $\nu=2$ and that $F_{(15,5)}'\ne 0$. Then
$F_{(15,5)}'=F_{(3,1)}^5$ (up to a non-zero scalar) and 
$$
\frac{F_{(3,1)}H_3}{F} + \frac{ F_{(3,1)}^5  H_1}{F} \in\bbk[\gt s^*] \ \text{ leading to } \
 \frac{H_3 + F_{(3,1)}^4  H_1}{F} \in\bbk[\gt s^*].  
$$
Modifying ${\bf h}_3$ and $H_3$ accordingly, we obtain a new $H_3$ with $\deg_{V_1}H_3\le 12$.  

Suppose now that $\nu=1$. If $F_{(15,5)}'\ne 0$, then we obtain
$\frac{F_{(3,1)}^5 H_1}{F}\in\bbk[\gt s^*]$, which cannot be the case. 
Thereby $F_{(15,5)}'=0$ and 
$$
\frac{F_{(3,1)}H_3}{F} + \frac{ F'_{(11,5)}  H_1}{F} + s_2\frac{ F_{(3,1)}^3 H_2}{F} \in\bbk[\gt s^*].  
$$
Since $2{\times}5=10<11$ and since $\psi_v(F_{(4,0)})=1$, 
the polynomial $F'_{(11,5)}$ is divisible by $F_{(3,1)}$, say $F'_{(11,5)}=F_{(3,1)} {\bf F}$. 
Now 
$$
\frac{H_3 +   {\bf  F} H_1 + s_2 F_{(3,1)}^2 H_2}{F} \in\bbk[\gt s^*].  
$$
This allows us to replace $H_3$, modifying ${\bf h}_3$ at the same time, 
by a polynomial of tri-degree $(2,8,8)$ keeping the property 
$\psi_v(H_3)={\bf h}_3$.   
\end{proof}

\begin{cl} Suppose that $\tilde{\gt s}=\tilde{\gt g}{\ltimes}V$ is given by the pair 
$(\Spin_{12}{\times}\SL_2, V_{\varphi_5}{\otimes}\bbk^2)$. Then $\tilde{\gt s}$ has ({\sf FA}) and
$\bbk[\tilde{\gt s}^*]^{\tilde{\gt s}}=\bbk[V^*]^{\tilde G}[H_1,H_2,H_3]$, where the bi-degrees of 
$H_i$ are $(2,8)$, $(2,12)$, $(2,16)$. 
\end{cl}
\begin{proof}
Let $\gt s=\gt g{\ltimes}V$ and $\bbk[\gt s^*]^S=\bbk[V^*]^G[H_1,H_2,H_3]$ be as in Theorem~\ref{th-D6-25}. 
Then $\tilde{\gt g}_x=\gt g_x$ for generic $x\in V^*$. 
Hence $\bbk[\tilde{\gt s}^*]^{\tilde S}\subset \bbk[\gt s^*]$ by Lemma~\ref{lm-sub}.  
According to the proof of Theorem~\ref{th-D6-25}, $H_1$ and $H_2$ are $\SL_2$-invariants, 
i.e., they are $\tilde S$-invariants, and also $\tilde H_3\not\in \gS^8(V)H_1\oplus\gS^4(V)H_2$ if 
$\tilde H_3\ne 0$.
At the same time the tri-degree of $\tilde H_3$ is $(2,7,9)$ if $\tilde H_3\ne 0$.
Combining these two observations, we see that $\tilde H_3=0$, $H_3$ is an $\SL_2$-invariant, and 
$\bbk[\tilde{\gt s}^*]^{\tilde{\gt s}}=\bbk[V^*]^{\tilde G}[H_1,H_2,H_3]$. Since $\bbk[V^*]^{\tilde G}$ is a polynomial
ring, the result follows. 
\end{proof}

\begin{prop}   \label{pos-red}
All the remaining cases marked with $`+'$ in Table~\ref{table:ort1} are indeed positive. 
\end{prop}
\begin{proof}
Making further use of Proposition~\ref{prop:trick}, we see that all the remaining 
cases are covered by reductions from $G$ of type $\GR{D}{5}$, see Diagrams~\eqref{eq:p-chain1}, \eqref{eq:p-chain2}, and also 
\[  
\xymatrix{
(\Spin_{10}, 3\varphi_1{+}\varphi_4) \ar[r] & (\Spin_{9}, 2\varphi_1{+}\varphi_4{+}2\bbk)  \ar[d]  & \\
     & (\Spin_8, \varphi_1{+}\varphi_3{+}\varphi_4{+}\bbk) \ar[r] & 
(\Spin_{7}, 2\varphi_3) , }  
\]    
where the initial pair is positive by Example~\ref{ex:7d}.
\end{proof}

\section{The classification for the symplectic algebra}
\label{sect:sp}

\noindent 
In this section, $G=\Sp_{2n}$. We classify the finite-dimensional rational representations $(G:V)$ 
such that $\mathsf{g.i.g.}(G:V)$ is infinite and 
the symmetric invariants of $\es=\g\ltimes V$ form a polynomial ring. The answer is given in Table~\ref{table-sp}. Surprisingly, all the possible candidates for
$\gt s=\gt g\ltimes V$ do have ({\sf FA}).   

Let $e\in\gt g$ be a nilpotent element and $\gt g_e\subset\gt g$ its centraliser. 
Then $\gt g_e$ has ({\sf FA}) by \cite{ppy}. 
This does not seem  to be relevant to our current task, but it is. 

The nilpotent element $e$ can be included into an $\gt{sl}_2$-triple 
$\{e,h,f\}\subset\gt g$ and this gives rise to the decomposition 
$\gt g=\bbk f \oplus e^{\perp}$, where $e^{\perp}$ is the subspace orthogonal to $e$ w.r.t. the Killing form of $\gt g$.  Let $\Delta_k\in\gS(\gt{sp}_{2n})$ be the sum of the  principal $k$-minors. We write the 
 highest $f$-component of $\Delta_k$ as ${^e\!}\Delta_k f^d $. Then $\{{^e\!}\Delta_k\mid k \text{ even, } 2\le k\le 2n\}$ is a set of the basic symmetric invariants of $\gt g_e$ \cite[Theorem\,4.4]{ppy}.

Let now $e$ be a minimal nilpotent element. Then  
$\gt g_e=\gt{sp}_{2n{-}2}{\ltimes}\gt{heis}_{n{-}1}$.  Restricting $H\in\gS(\gt g_e)^{\gt g_e}$ to
the hyperplane in $\g_e^*$, where $e=0$, we obtain a symmetric invariant of $\es:=\gt{sp}_{2n{-}2}{\ltimes}\bbk^{2n{-}2}$. 

Let $H_i$ be the restriction of ${^e\!}\Delta_{2i{+}2}$ to the hyperplane $e=0$.

\begin{lm}\label{m1}
The algebra of symmetric invariants of $\gt s=\gt{sp}_{2n{-}2}{\ltimes}\bbk^{2n{-}2}$ is freely generated 
by the polynomials $H_i$ as above with $1\le i\le n{-}1$.  
\end{lm}
\begin{proof}
Set $n'=n{-}1$. The group 
$G'=\Sp_{2n'}$ acts on $V^*\simeq V=\bbk^{2n'}$ with an open orbit, which consists of all 
non-zero vectors of $V^*$. Therefore $\gS(\gt s)^{\gt s}\simeq \gS(\gt h)^H$, where 
$$H=(\Sp_{2n'})_v=\Sp_{2n'{-}2}{\ltimes}\exp(\gt{heis}_{n'{-}1})$$ and $v\in V$ is non-zero. 
By a coincidence, $\gt h=\gt g'_{e'}$, where $e'\in\gt g'$ is a minimal nilpotent element. 
We have to show that $\psi_v(H_i)$ form a set  of the basic symmetric invariants of $\gt h$
for the usual restriction $\psi_v\!: \bbk[\gt s]^{\gt s}\to \bbk[(\gt g')^*+v]^{G'{\ltimes}\exp(V)}\simeq\gS(\gt h)^{\gt h}$.  

Note that the $f$-degree of each $\Delta_k$ with even $k$ is one, see \cite{ppy} and 
the matrix description of elements of $f+\gt g_e$ presented in  
Figure~\ref{Pic1}.
\begin{figure}[htb]
\begin{center}
\begin{tikzpicture}[scale= .6]
\draw (0,0)  rectangle (10,10);
\path[draw]  (1,0) -- (1,10); 
\path[draw]  (2,0) -- (2,10); 
\path[draw]  (0,8) -- (10,8);
\path[draw]  (0,9) -- (10,9); 

\path[fill=gray!40]  (1,0) -- (2,0) -- (2,8) -- (1,8)--cycle ;
\path[fill=gray!40]  (2,9) -- (10,9) -- (10,10) -- (2,10)--cycle ;

\draw (0.5,9.5)  node {\footnotesize {\color{darkblue}$0$}} ;
\draw (1.5,8.5)  node {\footnotesize {\color{darkblue}$0$}} ;
\draw (6,4)  node {\large{\color{redi}$\mathfrak{sp}_{2n-2}$}} ;
\draw (0.5,8.5)  node {\footnotesize {$1$}} ;
\draw (1.5,9.5)  node {\footnotesize {\color{redi}$c$}} ;
\draw (2.5,8.5)  node {\footnotesize {\color{darkblue}$0$}} ;
\draw (9.5,8.5)  node {\footnotesize {\color{darkblue}$0$}} ;
\draw (2.5,9.5)  node {\footnotesize {\color{redi}$\ast$}} ;
\draw (9.5,9.5)  node {\footnotesize {\color{redi}$\ast$}} ;
\draw (5.8,9.4)  node { {\color{darkblue}$\cdots$}} ;
\draw (0.5,7.5)  node {\footnotesize {\color{darkblue}$0$}} ;
\draw (0.5,0.5)  node {\footnotesize {\color{darkblue}$0$}} ;
\draw (1.5,7.5)  node {\footnotesize {\color{redi}$\ast$}} ;
\draw (1.5,0.5)  node {\footnotesize {\color{redi}$\ast$}} ;
\draw (5.8,8.4)  node { {\color{darkblue}$\cdots$}} ;
\draw (0.5,4.5)  node { {\color{darkblue}$\vdots$}} ;
\draw (1.5,4.5)  node { {\color{darkblue}$\vdots$}} ;
\end{tikzpicture}
\caption{Elements of $f+\gt g_e$.}   \label{Pic1}
\end{center}
\end{figure}
 Further, 
${^e\!}\Delta_{2i{+}2}$ is a sum $e\Delta_{2i}'+H_i$, 
where $\Delta_{2i}'\in\gS(\gt g')$. Choosing $v=(1,0,\ldots,0)^t$, one readily 
sees that $\psi_v(H_i)={^{e'}\!}\Delta_{2i}'$. This concludes the proof.  
\end{proof}

\begin{rmk}
We have a nice matryoshka-like structure. 
Starting from $\gt g_e$ with $\gt g=\gt{sp}_{2n{+}2}$ and restricting the symmetric invariants to the hyperplane $e=0$ one obtains the  
symmetric invariants of the semi-direct product 
$\gt{sp}_{2n}{\ltimes}\bbk^{2n}$. By passing to the stabiliser of a generic  point 
$x\in V^*$ with $V=\bbk^{2n}$, one comes back to $(\gt{sp}_{2n'})_{e'}$ with $n'=n{-}1$. 
And so on. 
\end{rmk}

Suppose now that $e\in\gt g$ is given by the partition $(2^m,1^{2n})$, $\gt g=\gt{sp}_{2m{+}2n}$. 
Then $\gt g_e=(\gt{so}_m{\oplus}\gt{sp}_{2n}){\ltimes}(\bbk^m{\otimes}\bbk^{2n}\oplus \gS^2\bbk^m)$ and 
the nilpotent radical of $\gt g_e$ is two-step nilpotent. Suppose that $m$ is odd. 
Set $Y:=\Ann(\gS^2\bbk^m)\subset \gt g_e^*$ and let $\tilde H_i$ be the restriction to 
$Y$ of ${^e\!}\Delta_k$ with $k=3m{+}2i{-}1$. 

\begin{lm}\label{m-odd} For $1\le i\le \left(n-\frac{m{-}1}{2}\right)$, we have $\tilde H_i\in\gS(\gt s)^{\gt s}$, where $\gt s=\gt{sp}_{2n}{\ltimes}m\bbk^{2n}$. 
\end{lm}
\begin{proof}
By the construction, each $\tilde H_i$ is $\gt g_e$-invariant.   Note that $\gt  g_e$ acts on 
$Y$ as the semi-direct product $(\gt{so}_m{\oplus}\gt{sp}_{2n}){\ltimes}\bbk^m{\otimes}\bbk^{2n}$. 
For each even $k$ with $k\ge 2m$, the $f$-degree of $\Delta_k$ is $m$ \cite{ppy}. For the corresponding 
$\tilde H_i$, this means that $\tilde H_i\in\gS(\gt s)$, see also Figure~\ref{Pic2}, where $C\in \gS^2\bbk^m$. 
\begin{figure}[htb]
\begin{center}
\begin{tikzpicture}[scale= .6]
\draw (0,0)  rectangle (10,10);
\path[draw]  (2,0) -- (2,10); 
\path[draw]  (4,0) -- (4,10); 
\path[draw]  (0,8) -- (10,8);
\path[draw]  (0,6) -- (10,6); 

\path[fill=gray!40]  (2,0) -- (4,0) -- (4,6) -- (2,6)--cycle ;
\path[fill=gray!40]  (4,8) -- (10,8) -- (10,10) -- (4,10)--cycle ;
\draw (7.5,3)  node {\large{\color{redi}$\mathfrak{sp}_{2n}$}} ;

\draw (1,9)  node {\large {\color{redi}$A$}} ;
\draw (1,7)  node {\large $I_m$};
\draw (3,9)  node {\large {\color{redi}$C$}} ;
\draw (3,7)  node {\large {\color{redi}$A$}} ;

\draw (9.5,8.5)  node {\footnotesize {\color{redi}$\ast$}} ;
\draw (9.5,9.5)  node {\footnotesize {\color{redi}$\ast$}} ;
\draw (4.5,8.5)  node {\footnotesize {\color{redi}$\ast$}} ;
\draw (4.5,9.5)  node {\footnotesize {\color{redi}$\ast$}} ;
\draw[dashed,darkblue]  (5.5,9.5) -- (8.7,9.5) ;
\draw[dashed,darkblue]  (4.5,9) -- (9.7,9) ;
\draw[dashed,darkblue]  (5.5,8.5) -- (8.7,8.5) ;  

\draw (9.5,6.5)  node {\footnotesize {\color{darkblue}$0$}} ;
\draw (9.5,7.5)  node {\footnotesize {\color{darkblue}$0$}} ;
\draw (4.5,6.5)  node {\footnotesize {\color{darkblue}$0$}} ;
\draw (4.5,7.5)  node {\footnotesize {\color{darkblue}$0$}} ;
\draw[dashed,darkblue]  (5.5,7.5) -- (8.7,7.5) ;
\draw[dashed,darkblue]  (4.5,7) -- (9.7,7) ;
\draw[dashed,darkblue]  (5.5,6.5) -- (8.7,6.5) ;  

\draw (0.5,5.5)  node {\footnotesize {\color{darkblue}$0$}} ;  %
\draw (1.5,5.5)  node {\footnotesize {\color{darkblue}$0$}} ;  
\draw (0.5,0.5)  node {\footnotesize {\color{darkblue}$0$}} ;   
\draw (1.5,0.5)  node {\footnotesize {\color{darkblue}$0$}} ;  
\draw[dashed,darkblue]  (0.5,1.1) -- (0.5, 4.9) ;
\draw[dashed,darkblue]  (1, 0.5) -- (1, 5.5) ;
\draw[dashed,darkblue]  (1.5,1.1) -- (1.5, 4.9) ;   

\draw (2.5,5.5)  node {\footnotesize {\color{redi}$\ast$}} ;  %
\draw (3.5,5.5)  node {\footnotesize {\color{redi}$\ast$}} ;  
\draw (2.5,0.5)  node {\footnotesize {\color{redi}$\ast$}} ;   
\draw (3.5,0.5)  node {\footnotesize {\color{redi}$\ast$}} ;  
\draw[dashed,darkblue]  (2.5,1.1) -- (2.5, 4.9) ;
\draw[dashed,darkblue]  (3, 0.5) -- (3, 5.5) ;
\draw[dashed,darkblue]  (3.5,1.1) -- (3.5, 4.9) ;   

\end{tikzpicture}
\caption{Elements of $f+\gt g_e\subset\gt{sp}_{2m{+}2n}$.}   \label{Pic2}
\end{center}
\end{figure}
\end{proof}

\begin{thm}\label{thm-sp}
All semi-direct products associated with pairs listed in Table~\ref{table-sp} have ({\sf FA}). 
\end{thm}
\begin{proof}
We begin with Item~1. \\ \indent
{\sf Suppose that  $m$ is even.} 
Set $\tilde G:=\Sp_{2n}{\times}\Sp_m$ and $\tilde S:=\tilde G{\ltimes}\exp(V)$.
Then $G\lhd \tilde G$ and $S\lhd \tilde S$. 
The Lie algebra $\tilde{\gt s}=\Lie\tilde S$ is the $\Z_2$-contraction of 
$\gt{sp}_{2n{+}m}$ related to the symmetric pair $(\gt{sp}_{2n{+}m},\gt{sp}_{2n}{\oplus}\gt{sp}_m)$. 
Let $\Delta_k\in\gS(\gt{sp}_{2n{+}m})$ be the sum of   
the  principal $k$-minors and let $\Delta_k^\bullet$ be the highest $V$-component of $\Delta_k$. 
The elements $\Delta_k^\bullet$ with even $k$, $2m<k\le 2n{+}m$, belong  to  a set  of the  
algebraically independent generators of $\gS(\tilde{\gt s})^{\tilde{\gt s}}$, see 
\cite[Theorem~4.5]{contr}.
 For a generic point $x\in V^*$, their restrictions 
$\Delta_k^\bullet|_{\tilde{\gt g}+x}$ form a generating set for the symmetric invariants of 
$(\gt{sp}_{2n})_x=\gt{sp}_{2n{-}m}$.  Hence $\Delta_k^\bullet\in \gS(\gt s)^S$ by 
Lemma~\ref{lm-sub}. 
According to \cite[Lemma~3.5({\sf ii})]{Y16}, these elements
$\Delta_k^\bullet$ (freely) generate $\gS(\gt s)^\gt s$ over $\bbk[V^*]^G$ if and only if their restrictions to $\gt g\times D$ are algebraically independent over $\bbk[D]^G$ for any $G$-invariant divisor 
$D\subset V^*$.  

In case $\Sp_{m}{\cdot}D$ is open in $V^*$, the restrictions of the elements 
$\Delta_k^\bullet$ to $\gt g+y$ are algebraically independent for a generic point $y\in D$. If 
$\Sp_{m}{\cdot}D$ is not open in $V^*$, then $D$ is $\tilde G$-invariant and the restrictions of $\Delta_k^\bullet$ to $\gt g\times D$ are algebraically independent over $\bbk[D]^{\tilde G}$ by   \cite[Lemma~3.5({\sf ii})]{Y16}
applied to $\tilde{\gt s}$.  If there is a non-trivial relation among these restrictions and 
not all the coefficients are $\tilde G$-invariant, then   
one can apply an element of $\tilde G$ to the relation and by taking a suitable linear combination obtain
a smaller non-trivial one.   Thus, 
a minimal non-trivial relation among the restrictions  must have 
$\tilde G$-invariant coefficients.  Hence  the restrictions of  $\Delta_k^\bullet$ to $\gt g\times D$  are also algebraically independent over $\bbk[D]^G$.   

{\sf Suppose now that $m$ is odd.} Consider the standard embedding 
$\gt{sl}_{2n}\subset \gt{sl}_{2n}\times \gt{sl}_{m}\subset \gt{sl}_{2n{+}m}$. 
The defining representation of $\Sp_{2n}$ in $\bbk^{2n}$ is self-dual. 
Therefore we can embed $V\simeq V^*$ into $m\bbk^{2n}\oplus m(\bbk^{2n})^*$ diagonally. 
This gives rise to $\gt s^*=\gt g\oplus V\subset \gt{sl}_{2n+m}$. 
Let $\Delta_k\in\gS(\gt{sl}_{2n{+}m})$ be the sum of the principal 
$k$-minors and  $\Delta_k^\bullet$ the 
highest $V$-component of the restriction ${\Delta_k}|_ {\gt s^*}$.   
Note that in case $m=1$, we have  $\Delta_k^\bullet=-H_i$, where $H_i$ is the same as in Lemma~\ref{m1}
and $k=2i{+}1$. For $m\ge 3$, $\Delta_k^\bullet$ is equal to $\pm \tilde H_i$, where $\tilde H_i$ is the same as in Lemma~\ref{m-odd} and $k=2m{+}2i{-}1$. Suppose that $m\ge 3$. 

Fix a $G$-stable  decomposition $V=V_1{\oplus}V_2$ with $V_1=\bbk^{2n}$.  Then there is the corresponding  decomposition 
$V^*=V_1^*\oplus V_2^*$. Choose a generic $v\in V_2^*$ and consider the usual restriction homomorphism 
$$
\psi_v\!: \bbk[\gt s^*]^S\to \bbk[\gt g{\oplus}V_1^*{+}v]^{G_v{\ltimes}\exp(V)}\simeq\gS(\gt g_v{\ltimes}V_1)^{G_v{\ltimes}\exp(V_1)}.
$$
Here $G_v=\Sp_{2n{-}m{+}1}$. Setting $n':=n-\frac{m-1}{2}$, we obtain 
$\gt g_v{\ltimes}V_1=(\gt{sp}_{2n'}{\ltimes}\bbk^{2n'})\oplus \bbk^{m{-}1}$. 
If $k=2m{+}2i{-}1$, then the restriction of $\Delta_k^\bullet$ to $\gt g{\oplus}V_1^*+v$ is equal to 
$cH_i$, where $c\in\bbk^{\!^\times}\!$ and $H_i$ is the same symmetric  invariant of 
$\gt{sp}_{2n'}{\ltimes}\bbk^{2n'}$ as in Lemma~\ref{m1}. 
 
The ring $\bbk[V^*]^G$ is freely generated by $\binom{m}{2}$ polynomials $F_j$ of degree $2$. 
We may (and will) assume that the first $m{-}1$ elements $F_j$ lie in 
$V_1{\otimes}V_2$ and that the remaining ones (freely) generate $\bbk[V_2^*]^G$. Then 
$\psi_v(F_j)\in\bbk$ for $j\ge m$ and $\left<\psi_v(F_j)\mid 1\le j\le m{-}1\right>_{\bbk}$ is the 
Abelian direct summand $\bbk^{m{-}1}$ of $\gt g_v{\ltimes}V_1$. We see that 
$F_1,\ldots,F_{m{-}1},\Delta_{2m{+}1}^\bullet ,\ldots,\Delta_{2n{+}m}^\bullet$ are 
algebraically independent over $\bbk[V_2^*]$. Hence 
$$\left\{F_j\mid 1\le j\le \binom{m}{2}\right\}\cup \{\Delta_k^\bullet \mid k \text{ odd, } 2m<k\le 2n{+}m\}$$
is a set of algebraically independent homogeneous invariants. Our goal is to prove that this 
is a generating set. 

There is a big open subset $U\subset V^*$
such that  $G_v$ is a generic isotropy group for $(G{:}V^*)$ for each $v\in U$. Here 
$G_v=(\Sp_{2n'})_e$ with $2n'=2n{-}m{+}1$ and $e\in\gt{sp}_{2n'}$ being a minimal nilpotent element. 
The algebra $\gt g_v$ has the ``codim--2" property by \cite{ppy} and hence 
$\gt s$ has the ``codim--2" property as well. 

Finally we calculate the sum of the degrees of the proposed generators.  
There are $\binom{m}{2}$ invariants of degree $2$, 
the minors $\Delta_k^\bullet$ are of degrees $2m{+}1$, $2m{+}3$, \ldots, $m{+}2n$.
Summing up 
$$
2\binom{m}{2}+\frac{1}{2}\left(n-\frac{m-1}{2}\right)\left(2n+3m+1\right)=\frac{1}{2}\ind\gt s+
n^2+\frac{n}{2}+nm=\frac{\ind\gt s+\dim\gt s}{2}.
$$
Applying Theorem~\ref{thm-codim2}, we can conclude that $\gS(\gt s)^{\gt s}$ is freely generated by the 
polynomials $F_j$ and $\Delta_k^\bullet$. 

Item~2  is a  $\Z_2$-contraction of $\SL_{2n}$, and this contraction is good, see \cite[Theorem~4.5]{contr}.

Item~4 can be covered by Theorem~\ref{V-rank-1} (or Lemma~\ref{l-A2}), this pair $(G,V^*)$ is of rank one. 
There is an open orbit $G{\cdot}y\subset D$, where $D$ stands for the zero set of the generator $F\in\bbk[V^*]^G$.  
A generic isotropy group for  $(G{:}V^*)$ is $\SL_3$,  $G_y$ is connected, and
 $\gt g_y$ is equal to 
$\tri\ltimes\gS^4\bbk^2$, see~\cite{igusa}.   This $\gt g_y$ is a good $\Z_2$-contraction of $\gt{sl}_3$ 
\cite{coadj}. 

Item~5 is covered by Example~\ref{ex-A1}.   

Item~6 is treated in \cite[Appendix~A]{md-ko}, there it is shown that this pair has  ({\sf FA}). 

The final challenge is to describe the symmetric invariants for item~3. A certain similarity with 
item~2 will help. 
Now $V=V_1{\oplus}V_2$ with $V_1=\bbk^{2n}$, $V_2=V_{\varphi_2}$. 
Set $\gt s_2:=\gt g{\ltimes}V_2$ (this is the semi-direct product in line 2). 
According to \cite{contr}, 
$\bbk[\gt s_2^*]^{\gt s_2}=\bbk[V_2^*]^G[{\bf h}_1,\ldots,{\bf h}_n]$, where each ${\bf h}_i$ is bi-homogeneous 
and $\deg_{\gt g}{\bf h}_i=2$.  In other words, ${\bf h}_i\in(\gS^2(\gt g){\otimes}\gS(V_2))^G$. 
In $\gS^2(V_1)$, there is a unique copy of $\gt g$, which gives rise to  embeddings 
$\iota\!: \gS^2(\gt g)\to \gt g{\otimes}\gS^2(V_1)$ and 
$$
\tilde\iota\!: (\gS^2(\gt g){\otimes}\gS(V_2))^G\to  (\gt g{\otimes}\gS^2(V_1){\otimes}\gS(V_2))^G.
$$
Set $H_i:=\tilde\iota({\bf h}_i)$. 

Each $H_i$ is a $G$-invariant by the construction. Next we check that it is also a $V$-invariant. 
Take a generic point $v\in V_2^*$. Then $\gt g_v$ is a direct sum of $n$ copies of $\gt{sl}_2$ and under 
$\gt g_v$ the space $V_1$ decomposes into a direct sum of $n$ copies of $\bbk^2$. 
The restriction of ${\bf h}_i$ to $\gt g+v$ is an element of $\gS^2(\gt g_v)^{\gt g_v}\subset\gS^2(\gt g)\subset\gt g{\otimes}\gt g$. If we regard this restriction as a bi-linear function on $\gt g{\otimes}\gt g$, 
then its value on $(A,B)$ for $A,B\in\gt g$ can be calculated as follows.   
From each matrix we cut the $\gt{sl}_2$ pieces $A_j,B_j$, $1\le j\le n$,
corresponding to the $\gt{sl}_2$ summands of $\gt g_v$ and  take a linear combination 
$\sum \alpha_{i,j} \tr(A_jB_j)$. 
With a slight abuse of notation we set ${\bf h}_i(A,B,v):=\sum \alpha_{i,j} \tr(A_jB_j)$. 

The restriction of $H_i$ to $\g\oplus V_1^*+v$ is an element of $(\gt g\otimes\gS^2(V_1))^{\g_v}$. 
Take $\xi\in V_1^*$. Let $B(\xi)\in\gt g$ be the projection of $\xi^2$ to $\gt g\subset \gS^2(V_1)$. 
Then 
$$
     H_i(A+\xi+v)={\bf h}_i(A,B(\xi),v) .
$$ 
Write $\xi=\xi_1+\ldots+\xi_n$, where each $\xi_j$ lies in its 
$\gt g_v$-stable copy of $\bbk^2$. Then $\xi_j{\otimes}\xi_k$ with $j\ne k$ is orthogonal to 
$\gt g_v\subset\gt g\subset\gS^2(V_1)$.  Furthermore, $\tr(A_jB(\xi)_j)=\det(\xi_j|A_j\xi_j)$. 
Therefore 
$$
     H_i(A+\xi + v)=\sum \alpha_{i,j} \det(\xi_j|A_j\xi_j) .
$$ 
We see that 
$H_i|_{\gt g{\oplus}V_1^*+v}$ lies in $\gS(\gt g_v{\ltimes}V_1)$ and therefore is a $V_2$-invariant \cite{Y16}. 
Moreover, this restriction is a $V_1$-invariant by \cite{Y}. Since these assertions  hold for a generic vector $v\in V_2^*$, each $H_i$ is a $V$-invariant. 
From the case of $\gt s_2$, we know that the matrix $(\alpha_{i,j})$ is non-degenerate. 
Hence the invariants $H_i$ are algebraically independent over $\bbk(V_2^*)$. 
Note that $\bbk[V_2^*]^G=\bbk[V^*]^G$.  Further, $\deg H_i=\deg{\bf h}_i+1$. If we sum over all 
(suggested) generators, then the  result is $(\dim\gt s_2+\ind\gt s_2)/2+n$ and this is 
exactly $(\dim\gt s+\ind\gt s)/2$. 

In order to use Theorem~\ref{thm-codim2}, it remains to prove that $\gt s$ has the 
``codim--2" property. Let $D\subset V_2^*$ be  a $G$-invariant divisor
and let $y\in D$ be a generic point. If $G_y \ne (\SL_2)^n$, then  $G_y=(\SL_2)^{n{-}2}\times (\SL_2{\ltimes}\exp(\gS^2\bbk^2))$. 
In particular, $\dim(G{\cdot}y)=\dim V_2-(n{-}1)$. 
If $\q$ is the Lie algebra of $Q=\SL_2{\ltimes}\exp(\gS^2\bbk^2)$, then   $\q=\gt{sl}_2{\ltimes}\gt{sl}_2^{\sf ab}$.
We have 
$$
  G_y{\ltimes}\exp(V_1)=(\SL_2{\ltimes}\exp(\bbk^2))^{n{-}2}\times (Q{\ltimes}\exp({\bbk}^4))
$$ 
and $\q{\ltimes}\bbk^4=\gt{sl}_2{\ltimes}((\bbk^2{\oplus}\gS^2\bbk^2)\oplus\bbk^2)$ 
with the unique  non-zero commutator $[\bbk^2,\gS^2\bbk^2]=\bbk^2$. 
An easy computation shows that 
$\ind(\q{\ltimes}\bbk^4)=2$. Thereby $\ind(\gt g_y{\ltimes}V_1)=n$ and hence $\gt g{\oplus} V_1^* {\times} D\cap \gt s^*_{\sf reg}\ne \varnothing$, cf. \cite[Eq.~(3${\cdot}$2)]{Y16}. The Lie algebra $\gt s$ does have the ``codim--2" property. 
\end{proof}

{\bf Acknowledgements.}  This work was completed while first author was a guest of the Max-Planck-Institut f\"ur Mathematik (Bonn).

\end{document}